\documentclass[leqno,12pt]{amsart} 

\usepackage{amssymb} 
\usepackage{amsmath}
\usepackage{amsthm}
\usepackage{amscd}
\usepackage{bm}
\usepackage{graphicx,psfrag,wrapfig}
\theoremstyle{plain} 
\newtheorem{theorem}{Theorem}[section]  
\newtheorem{lemma}[theorem]{Lemma}
\newtheorem{corollary}[theorem]{Corollary}
\newtheorem{proposition}[theorem]{Proposition}

\newtheorem{sublemma}[theorem]{Sublemma}
\theoremstyle{definition} 

\newtheorem{definition-lemma}[theorem]{Definition-Lemma}
\newtheorem{remark}[theorem]{Remark}
\newtheorem{example}[theorem]{Example}

\newcommand{\affine}{\mathbb{C}}
\newcommand{\const}{\mathrm{const}}
\newcommand{\norm}[1]{\left|\!\left|#1\right|\!\right|}
\newcommand{\supp}{\mathrm{supp}}

\newcommand{\diam}{\mathrm{Diam}}
\newcommand{\moduli}{\mathcal{M}}
\newcommand{\widim}{\mathrm{Widim}}
\newcommand{\dist}{\mathrm{dist}}
\newcommand{\area}{\mathrm{Area}}

\begin{document}

\title[Brody curves and mean dimension]
{Brody curves and mean dimension} 

\author[S. Matsuo and M. Tsukamoto]{Shinichiroh Matsuo and Masaki Tsukamoto}



\thanks{Shinichiroh Matsuo was supported by Grant-in-Aid for JSPS fellows (23$\cdot$149), and
Masaki Tsukamoto was supported by Grant-in-Aid for Young Scientists (B) (21740048).}

\date{\today}

\maketitle

\begin{abstract}
We study the mean dimensions of the spaces of Brody curves.
In particular we give the formula of the mean dimension of the space of Brody curves in 
the Riemann sphere.
A key notion is a non-degeneracy of Brody curves introduced by Yosida (1934).
We develop a deformation theory of non-degenerate Brody curves
and apply it to the calculation of the mean dimension.
Moreover we show that 
there are sufficiently many non-degenerate Brody curves.
\end{abstract}


\section{Introduction} \label{section: introduction}
\subsection{Main results}
Let $z=x+y\sqrt{-1}\in \affine$ be the standard coordinate in the complex plane $\affine$.
Let $f = [f_0:f_1:\dots:f_N]:\affine \to \affine P^N$ be a holomorphic map ($f_i$: holomorphic function).
We define $|df|(z)\geq 0$ by 
\[ |df|^2(z) := \frac{1}{4\pi}\Delta\log(|f_0|^2+|f_1|^2+\dots+|f_N|^2) \quad 
   \left(\Delta := \frac{\partial^2}{\partial x^2}+\frac{\partial^2}{\partial y^2}\right).\]
$|df|(z)$ is classically called a spherical derivative. It evaluates the dilatation of the map $f$ 
with respect to the Euclidean metric on $\affine$ and the Fubini-Study metric on $\affine P^N$.
(See the equation (\ref{eq: spherical derivative is dilatation}) in Section 
\ref{subsection: perturbation of the Hermitian metric}.)

A holomorphic map $f:\affine \to \affine P^N$ is called a Brody curve (\cite{Brody}) if it satisfies 
$|df|(z)\leq 1$ for all $z\in \affine$.
Let $\moduli(\affine P^N)$ be the space of Brody curves in $\affine P^N$.
It is endowed with the compact-open topology (the topology of uniform convergence on compact subsets):
A sequence of Brody curves $\{f_n\}\subset \moduli(\affine P^N)$ converges to $f\in \moduli(\affine P^N)$
if and only if for any compact subset $K\subset \affine$ we have 
$\sup_{z\in K}d(f_n(z),f(z))\to 0$ as $n\to \infty$.
($d(\cdot, \cdot)$ is the distance on $\affine P^N$ with respect to the Fubini-Study metric.)
$\moduli(\affine P^N)$ is an infinite dimensional compact metrizable space, and 
it admits
the following continuous $\affine$-action.
\[ \moduli(\affine P^N)\times \affine \to \moduli(\affine P^N), \quad (f(z), a)\mapsto f(z+a).\]
 The main purpose of the paper is to study the mean dimension $\dim(\moduli(\affine P^N):\affine)$ 
of this infinite dimensional dynamical system.
(Mean dimension is an invariant of topological dynamical systems introduced by Gromov \cite{Gromov}.
We review its definition in Section \ref{subsection: review of mean dimension}.)

Let $f:\affine \to \affine P^N$ be a Brody curve.
We define an energy density $\rho(f)$ by setting 
\begin{equation} \label{eq: definition of energy density}
 \rho(f) := \lim_{R\to \infty}\frac{1}{\pi R^2}\left(\sup_{a\in \affine}
    \int_{|z-a|< R} |df|^2 dxdy\right).
\end{equation}
(This limit always exists. See Section \ref{subsection: energy density}.)
We define the Nevanlinna-Shimizu-Ahlfors characteristic function $T(r,f)$ $(r\geq 1)$ by
\[ T(r, f) := \int_1^r\left(\int_{|z|<t}|df|^2dxdy\right)\frac{dt}{t}.\]
From the Brody condition $|df|\leq 1$, we have $T(r, f)\leq \pi r^2/2$.
We define $\rho_{\mathrm{NSA}}(f)$ by 
\[ \rho_{\mathrm{NSA}}(f) := \limsup_{r\to \infty}\frac{2}{\pi r^2}T(r,f).\]
It is easy to see
$\rho_{\mathrm{NSA}}(f)\leq \rho(f)$.

Let $\rho(\affine P^N)$ be the supremum of $\rho(f)$ over $f\in \moduli(\affine P^N)$, and 
let $\rho_{\mathrm{NSA}}(\affine P^N)$ be the supremum of 
$\rho_{\mathrm{NSA}}(f)$ over $f\in \moduli(\affine P^N)$.
We know (see Section \ref{subsection: energy density})
\[ 0 < \rho_{\mathrm{NSA}}(\affine P^N) \leq \rho(\affine P^N) <1 .\]

The main result of this paper is the following:
\begin{theorem} \label{main theorem}
\[ 2(N+1)\rho(\affine P^N)\leq \dim(\moduli(\affine P^N):\affine)\leq 4N\rho_{\mathrm{NSA}}(\affine P^N).\]
\end{theorem}
\begin{corollary} \label{main corollary}
\[ \dim(\moduli(\affine P^1):\affine) = 4\rho(\affine P^1) = 4\rho_{\mathrm{NSA}}(\affine P^1).\]
\end{corollary}
From Theorem \ref{main theorem}, 
$4\rho(\affine P^1)\leq \dim(\moduli(\affine P^1):\affine)\leq 4\rho_{\mathrm{NSA}}(\affine P^1)$.
Since $\rho_{\mathrm{NSA}}(\affine P^1) \leq \rho(\affine P^1)$, we get the corollary.

The formula $\dim(\moduli(\affine P^1):\affine) = 4\rho_{\mathrm{NSA}}(\affine P^1)$ was conjectured in 
\cite[p. 1643, (4)]{Tsukamoto-ETDS}.
This formula is very surprising (at least for the authors) because the definitions of the left-hand-side and
the right-hand-side are totally different.

The upper bound $\dim(\moduli(\affine P^N):\affine)\leq 4N\rho_{\mathrm{NSA}}(\affine P^N)$ was already
proved in \cite[Theorem 1.5]{Tsukamoto-Nagoya} by using the Nevanlinna theory.
(Remark:  We used the notation $e(f)$ for $\rho_{\mathrm{NSA}}(f)$ in \cite{Tsukamoto-Nagoya}.)
The purpose of the present paper is to prove the lower bound 
$\dim(\moduli(\affine P^N):\affine)\geq 2(N+1)\rho(\affine P^N)$.

\subsection{Non-degenerate Brody curves} \label{subsection: non-degenerate Brody curves}
For $a\in \affine$ and $r>0$ we set $D_r(a) := \{z\in \affine|\, |z-a|\leq r\}$.
The following is a key-notion of the paper.
This notion was first introduced by Yosida \cite{Yosida}. 
(Gromov \cite[p. 399]{Gromov} also discussed it in a more general situation.
See also Eremenko \cite[Section 4]{Eremenko} and Remark \ref{remark: yosida function} below.)
\begin{definition-lemma} \label{defintion-lemma: non-degenerate}
Let $f:\affine \to \affine P^N$ be a Brody curve.
Then the following two conditions are equivalent.

\noindent 
(i) Any constant curve does not belong to the closure of the $\affine$-orbit of $f$.
(In other words, for any sequence of complex numbers $\{a_n\}_{n\geq 1}$, the sequence of 
Brody curves $\{f(z+a_n)\}_{n\geq 1}$ does not converge to a constant curve.)

\noindent 
(ii) There exist $\delta>0$ and $R>0$ such that for all $a\in \affine$ we have 
 $\norm{df}_{L^\infty(D_R(a))}\geq \delta$.

$f$ is said to be non-degenerate if it satisfies one of (and hence both) the above conditions.
\end{definition-lemma}
\begin{proof}
The following argument is given in \cite{Yosida}.
Suppose that the condition (ii) fails.
Then for any $n\geq 1$ there is $a_n\in \affine$ such that $\norm{df}_{L^\infty(D_1(a_n))}\leq 1/n$.
Taking a subsequence, we can assume that the sequence $\{f(z+a_n)\}_{n\geq 1}$ converges to a
Brody curve $g(z)$. Then $\norm{dg}_{L^\infty(D_1(0))} =0$.
This implies that $g$ is a constant curve.

Suppose the condition (ii) holds.
Let $\{a_n\}_{n\geq 1}$ be a sequence of complex numbers.
If $\{f(z+a_n)\}_{n\geq 1}$ converges to $g(z)$, then $\norm{dg}_{L^\infty(D_R(0))}\geq \delta$.
Hence $g(z)$ is not a constant curve. This proves the condition (i).
\end{proof}
\begin{remark}  \label{remark: yosida function}
The above argument also proves that the conditions in Definition-Lemma \ref{defintion-lemma: non-degenerate}
are equivalent to the following:

\noindent 
(ii') For any $R>0$ there exists $\delta>0$ such that for all $a\in \affine$ we have 
$\norm{df}_{L^\infty(D_R(a))}\geq \delta$.

Yosida \cite[Theorem 4]{Yosida} proved (i) $\Leftrightarrow$ (ii') for the case of $N=1$.
In \cite{Yosida} Brody curves $f:\affine \to \affine P^1$ satisfying (i) are called 
meromorphic functions of 1st category.
In Eremenko \cite[Section 4]{Eremenko} Brody curves $f:\affine \to \affine P^N$ satisfying (i) are called 
binormal curves.
Gromov \cite[p. 399]{Gromov} used the terminology ``uniformly nondegenerate''.
\end{remark}
\begin{example} \label{example: non-degeneate Brody curve}
$f(z)=e^z\in \moduli(\affine P^1)$ is a degenerate (i.e. not non-degenerate) Brody curve.
A non-constant elliptic function $f(z)\in \moduli(\affine P^1)$ is a non-degenerate Brody curve.
\end{example}
In our viewpoint, non-degenerate Brody curves are ``non-singular points" of the space $\moduli(\affine P^N)$,
and they behave very nicely for the calculation of the mean dimension:
\begin{theorem} \label{thm: mean dimension around non-degenerate curves}
Let $f:\affine \to \affine P^N$ be a non-degenerate Brody curve with $\norm{df}_{L^\infty(\affine)}<1$.
Then 
\[ \dim(\moduli(\affine P^N):\affine)\geq 2(N+1)\rho(f).\]
\end{theorem}
The following theorem means that there are ``sufficiently many'' non-degenerate Brody curves:
\begin{theorem}\label{thm: there are many non-degenerate curves}
Let $f:\affine \to \affine P^N$ be a holomorphic map with $\norm{df}_{L^\infty(\affine)} <1$.
Then for any $\varepsilon>0$ there exists a non-degenerate Brody curve $g:\affine \to \affine P^N$ 
satisfying $\norm{dg}_{L^\infty(\affine)} < 1$ and $\rho(g)\geq \rho(f)-\varepsilon$.
\end{theorem}
\begin{proof}[Proof of Theorem \ref{main theorem}, 
assuming Theorems \ref{thm: mean dimension around non-degenerate curves} and 
\ref{thm: there are many non-degenerate curves}]
The upper bound $\dim(\moduli(\affine P^N):\affine)\leq 4N\rho_{\mathrm{NSA}}(\affine P^N)$
was already proved in \cite[Theorem 1.5]{Tsukamoto-Nagoya}.
Here we prove the lower bound.
Let $f:\affine \to \affine P^N$ be a Brody curve.
Let $0<c<1$ and set $f_c(z)=f(cz)$.
Then $|df_c|(z)= c|df|(cz)$ and $\rho(f_c) = c^2\rho(f)$.
Since $\norm{df_c}_{L^\infty(\affine)}\leq c < 1$, 
we can apply Theorem \ref{thm: there are many non-degenerate curves} to $f_c$.
Then for any $\varepsilon>0$ 
there exists a non-degenerate Brody curve $g:\affine \to \affine P^N$ satisfying 
$\norm{dg}_{L^\infty(\affine)}<1$ and $\rho(g)\geq \rho(f_c)-\varepsilon= c^2\rho(f)-\varepsilon$.
By Theorem \ref{thm: mean dimension around non-degenerate curves} 
\[ \dim(\moduli(\affine P^N):\affine)\geq 2(N+1)\rho(g) \geq 2(N+1)(c^2\rho(f)-\varepsilon).\]
Let $\varepsilon\to 0$ and $c\to 1$. We get 
$\dim(\moduli(\affine P^N):\affine)\geq 2(N+1)\rho(f)$.
Taking the supremum over $f\in \moduli(\affine P^N)$, we get 
$\dim(\moduli(\affine P^N):\affine)\geq 2(N+1)\rho(\affine P^N)$.
\end{proof}

\section{Some preliminaries}

\subsection{Review of mean dimension} \label{subsection: review of mean dimension}
In this subsection we review the definition of mean dimension.
For the detail, see Gromov \cite{Gromov} and Lindenstrauss-Weiss \cite{Lindenstrauss-Weiss}.
(For some related works, see also Lindenstrauss \cite{Lindenstrauss} and
 Gournay \cite{Gournay-thesis, Gournay, Gournay-holder, Gournay-neumann}.)

Let $(X,d)$ be a compact metric space, and let $Y$ be a topological space.
Let $\varepsilon>0$.
A continuous map $f:X\to Y$ is called an $\varepsilon$-embedding if
$\diam f^{-1}(y)\leq \varepsilon$ for all $y\in Y$.
Here $\diam f^{-1}(y)$ is the supremum of $d(x_1, x_2)$ over $x_1, x_2\in f^{-1}(y)$.
We define $\widim_\varepsilon(X,d)$ as the minimum integer $n\geq 0$ such that there are
an $n$-dimensional polyhedron $P$ and an $\varepsilon$-embedding $f:X\to P$.

For example, let $X = [0,1]\times [0,\varepsilon]$ with the Euclidean distance.
Then the projection $\pi:X\to [0,1]$ is an $\varepsilon$-embedding, and we have 
$\widim_\varepsilon (X, \mathrm{Euclid}) = 1$.
The following example is very important in the later argument.
This was given by Gromov \cite[p. 333]{Gromov}.
(For the detailed proof, see Gournay \cite[Lemma 2.5]{Gournay} or 
Tsukamoto \cite[Appendix]{Tsukamoto-ETDS}.)
\begin{example} \label{example: widim of Banach ball}
Let $V$ be a finite dimensional Banach space over $\mathbb{R}$, 
and set $B_r(V) := \{x\in V| \norm{x}\leq r\}$ for $r>0$.
For $0<\varepsilon <r$, 
\[ \widim_\varepsilon (B_r(V), \norm{\cdot}) =\dim V.\]
Here we consider the norm distance on $B_r(V)$.
\end{example}
For a subset $\Omega\subset \affine$ and $r>0$, we define $\partial_r\Omega$ as the set of 
$a\in \affine$ satisfying $D_r(a)\cap \Omega\neq \emptyset$ 
and $D_r(a)\cap (\affine\setminus \Omega)\neq \emptyset$.
Let $\Omega_n$ $(n\geq 1)$ be a sequence of bounded Borel subsets of $\affine$.
It is called a F{\o}lner sequence if for all $r>0$
\[ \frac{\area(\partial_r\Omega_n)}{\area(\Omega_n)} \to 0 \quad (n\to \infty).\]
For example, the sequence $\Omega_n := D_n(0)$ is a F{\o}lner sequence.
The sequence $\Omega_n:= [0,n]\times [0,n]$ is also F{\o}lner.
We need the following ``Ornstein-Weiss lemma''. 
(For the proof, see Gromov \cite[pp. 336-338]{Gromov}.)
\begin{lemma} \label{lemma: Ornstein-Weiss lemma}
Let $h: \{\text{bounded Borel subsets of }\,\affine\}\to \mathbb{R}_{\geq 0}$ be a map satisfying 
the following three conditions.

\noindent 
(i) If $\Omega_1\subset \Omega_2$, then $h(\Omega_1)\leq h(\Omega_2)$.

\noindent 
(ii) $h(\Omega_1\cup \Omega_2) \leq h(\Omega_1)+h(\Omega_2)$.

\noindent 
(iii) For any $a\in \affine$ and any bounded Borel subset $\Omega \subset \affine$, we have 
$h(a+\Omega)=h(\Omega)$ where $a+\Omega:=\{a+z\in \affine| z\in \Omega\}$.

Then for any F{\o}lner sequence $\Omega_n$ $(n\geq 1)$ in $\affine$, the limit of the sequence 
\[ \frac{h(\Omega_n)}{\area(\Omega_n)} \quad (n\geq 1) \]
exists, and its value is independent of the choice of a F{\o}lner sequence.
\end{lemma}
Suppose that the Lie group $\affine$ continuously acts on a compact metric space $X$.
Here we don't assume that the distance is invariant under the group action.
For a subset $\Omega \subset \affine$, we define a new distance $d_{\Omega}$ on $X$ by 
\[ d_{\Omega}(x,y) :=\sup_{a\in \Omega}d(a.x,a.y).\]
It is easy to see that the map $\Omega\mapsto \widim_\varepsilon(X,d_{\Omega})$
satisfies the three conditions in Lemma \ref{lemma: Ornstein-Weiss lemma} for each $\varepsilon>0$.
So we define a mean dimension $\dim(X:\affine)$ by 
\[ \dim(X:\affine) :=
    \lim_{\varepsilon\to +0}
\left(\lim_{n\to \infty} \frac{\widim_\varepsilon(X,d_{\Omega_n})}{\area(\Omega_n)}\right) \]
where $\Omega_n$ $(n\geq 1)$ is a F{\o}lner sequence in $\affine$.
The value of the mean dimension $\dim(X:\affine)$ is independent of the choice of a F{\o}lner sequence, 
and it is a topological invariant. 
(That is, it is independent of the choice of a distance on $X$ compatible with the 
topology.)
For example, we have 
\begin{equation} \label{eq: several definition of mean dimension}
 \begin{split}
 \dim(X:\affine) &= 
  \lim_{\varepsilon\to +0}
   \left(\lim_{R\to \infty} \frac{\widim_\varepsilon(X, d_{D_R(0)})}{\pi R^2}\right)\\
  &= \lim_{\varepsilon\to +0}
   \left(\lim_{R\to \infty}\frac{\widim_\varepsilon(X,d_{[0,R]\times [0,R]})}{R^2}\right).
 \end{split}
\end{equation}

\subsection{Energy density}   \label{subsection: energy density}

Here we explain some basic properties of the energy density $\rho(f)$ introduced in 
(\ref{eq: definition of energy density}).
Let $f:\affine\to \affine P^N$ be a Brody curve.
Then the map 
\[ \Omega \mapsto \sup_{a\in \affine} \int_{a+\Omega} |df|^2dxdy \]
clearly satisfies the three conditions in Lemma \ref{lemma: Ornstein-Weiss lemma},
where $\Omega\subset \affine$ is a bounded Borel subset.
Therefore we can define the energy density $\rho(f)$ by
\[ \rho(f) := \lim_{n\to \infty} \frac{1}{\area(\Omega_n)}
    \left(\sup_{a\in \affine}\int_{a+\Omega_n}|df|^2dxdy\right), \]
where $\Omega_n$ $(n\geq 1)$ is a F{\o}lner sequence in $\affine$.
In particular, we have 
\begin{equation}  \label{eq: several definition of energy density}
  \begin{split}
   \rho(f) &= \lim_{R\to \infty} \frac{1}{\pi R^2}\left( \sup_{a\in \affine}\int_{|z-a|<R} |df|^2dxdy\right)\\
   &= \lim_{R\to \infty}\frac{1}{R^2}
    \left(\sup_{a,b\in \mathbb{R}}\int_{[a,a+R]\times [b,b+R]} |df|^2dxdy\right).
  \end{split}
\end{equation}
From this we get 
\[ \rho(f) \geq \limsup_{R\to \infty} \frac{1}{\pi R^2}\int_{|z|<R}|df|^2 dxdy \geq 
    \limsup_{R\to \infty} \frac{2}{\pi R^2} T(R,f) =: \rho_{\mathrm{NSA}}(f).\]
If $f$ is elliptic (i.e. there is a lattice $\Lambda\subset \affine$ such that 
$f(z+\lambda)=f(z)$ for all $\lambda\in \Lambda$), then 
\begin{equation*}
   \rho(f) = \limsup_{R\to \infty} \frac{1}{\pi R^2}\int_{|z|<R}|df|^2 dxdy 
   = \rho_{\mathrm{NSA}}(f)
  = \frac{1}{\area(\affine/\Lambda)}\int_{\affine/\Lambda} |df|^2dxdy.
\end{equation*}

In the paper \cite{Tsukamoto-packing} we studied the quantity 
\[ \limsup_{R\to \infty} \frac{1}{\pi R^2}\int_{|z|<R}|df|^2 dxdy. \]
Some methods and results in \cite{Tsukamoto-packing} can be also applied to $\rho(f)$.
For example, from \cite[Proposition 2.6, Proposition 3.1]{Tsukamoto-packing}
(Proposition 3.1 in \cite{Tsukamoto-packing} follows from a result of Calabi 
\cite[Thoerem 8]{Calabi},),
there exists $0<c(N)<1$ such that for all Brody curves $f:\affine \to \affine P^N$ and 
all $a,b\in \mathbb{R}$
\[ \int_{[a,a+1]\times [b,b+1]} |df|^2 dxdy \leq c(N).\]
Hence 
\[ \rho(\affine P^N) = \sup_{f\in \moduli(\affine P^N)}\rho(f) \leq c(N) <1.\]
Moreover, from \cite[Proposition 5.10]{Tsukamoto-packing}, 
there exists $r>0$ such that for all Brody curves
$f:\affine \to \affine P^1$ and all $a,b\in \mathbb{R}$
\[ \frac{1}{r^2} \int_{[a,a+r]\times [b,b+r]} |df|^2 dxdy  \leq 1- 10^{-100}.\]
Hence we get an explicit (but very rough) bound:
\[ \rho(\affine P^1) \leq 1-10^{-100}.\]

In the paper \cite[Section 1.2]{Tsukamoto-ETDS}
we constructed an elliptic function $f:\affine \to \affine P^1$ such that 
$f$ is a Brody curve and 
\[ \rho(f) = \rho_{\mathrm{NSA}}(f) = \frac{2\pi}{\sqrt{3}}\left(\int_1^\infty \frac{dx}{\sqrt{x^3-1}}\right)^{-2}
   = 0.6150198678198\dots.\]
Hence 
\[ \frac{2\pi}{\sqrt{3}}\left(\int_1^\infty \frac{dx}{\sqrt{x^3-1}}\right)^{-2} \leq \rho(\affine P^1)
    \leq 1 -10^{-100}.\]
The authors think that it is very wonderful if the first inequality is an equality.

It is very difficult to determine the value of $\rho(\affine P^N)$, but
we have the following clear result on its asymptotic behavior:
The sequence $\rho(\affine P^N)$ $(N\geq 1)$ is a non-decreasing sequence, and 
from \cite[Theorem 1.5]{Tsukamoto-packing}, we have
\[ \lim_{N\to \infty} \rho(\affine P^N) =1.\]
Moreover the proof of \cite[Theorem 1.5]{Tsukamoto-packing} also shows
\[ \lim_{N\to \infty} \rho_{\mathrm{NSA}}(\affine P^N) =1.\]

\section{Proof of Theorem \ref{thm: mean dimension around non-degenerate curves}}
\label{section: proof of Theorem mean dimension around non-degenerate curves}

In this section we prove Theorem \ref{thm: mean dimension around non-degenerate curves}
assuming Propositions \ref{prop: result of deformation} and \ref{prop: structure of H_f} below.
Theorem \ref{thm:  there are many non-degenerate curves} will be proved in Section \ref{section: infinite gluing}.
Let $T\affine P^N$ be the tangent bundle of $\affine P^N$.
It naturally admits a structure of a holomorphic vector bundle.
We consider the Fubini-Study metric on it.
Let $f:\affine\to \affine P^N$ be a Brody curve,
and let $f^*T\affine P^N$ be the pull-back of $T\affine P^N$ by $f$.
$f^*T\affine P^N$ is a holomorphic vector bundle over the complex plane $\affine$,
and its Hermitian metric is given by the pull-back of the Fubini-Study metric.
Let $H_f$ be the space of holomorphic sections $u:\affine \to f^*T\affine P^N$
satisfying $\norm{u}_{L^\infty(\affine)}<+\infty$.
$(H_f,\norm{\cdot}_{L^\infty(\affine)})$ is a complex Banach space (possibly infinite dimensional).
We set $B_r(H_f) :=\{u\in H_f|\, \norm{u}_{L^\infty(\affine)}\leq r\}$ for $r\geq 0$.
\begin{proposition}  \label{prop: result of deformation}
Let $f:\affine \to \affine P^N$ be a non-degenerate Brody curve with $\norm{df}_{L^\infty(\affine)}<1$.
Then there exist $\delta>0$ and a map 
\[ B_\delta(H_f)\to \moduli(\affine P^N), \quad u\mapsto f_u, \]
satisfying the following two conditions:

\noindent 
(i) $f_0=f$.

\noindent 
(ii) For all $u, v\in B_\delta(H_f)$ and $z\in \affine$
\[ \left|d(f_u(z), f_v(z))-|u(z)-v(z)|\right|  \leq  \frac{1}{8}\norm{u-v}_{L^\infty(\affine)}.\]
Here $d(\cdot,\cdot)$ is the distance on $\affine P^N$ defined by the Fubini-Study metric, and 
$|u(z)-v(z)|$ is the fiberwise norm of $f^*T\affine P^N$.
\end{proposition}
Let $R>0$ and $\Lambda\subset \affine$.
$\Lambda$ is said to be an $R$-square if $\Lambda = [a,a+R]\times [b,b+R]$ for some $a,b\in \mathbb{R}$.
\begin{proposition}  \label{prop: structure of H_f}
Let $f:\affine \to \affine P^N$ be a non-degenerate Brody curve.
Then for any $R$-square $\Lambda\subset \affine$ with $R>2$ there 
exists a finite dimensional complex subspace $V\subset H_f$ satisfying the following two conditions:

\noindent 
(i) 
\[ \dim_\affine V\geq (N+1)\int_\Lambda |df|^2dxdy - C_f R.\]
Here $C_f$ is a positive constant depending only on $f$ (and independent of $R$, $\Lambda$).

\noindent 
(ii) For all $u\in V$ we have $\norm{u}_{L^\infty(\affine)}\leq 2\norm{u}_{L^\infty(\Lambda)}$.
\end{proposition}
Propositions \ref{prop: result of deformation} and \ref{prop: structure of H_f} will be proved later
(Sections \ref{section: deformation theory} and \ref{section: study of H_f}.)
Here we prove Theorem \ref{thm: mean dimension around non-degenerate curves}, assuming 
them.
\begin{proof}[Proof of Theorem \ref{thm: mean dimension around non-degenerate curves}]
We define a distance on $\moduli(\affine P^N)$ by 
\[ \dist(g,h) := \sum_{n=0}^\infty \frac{1}{10^n}\sup_{|z|\leq n}d(g(z),h(z)),
    \quad (g,h\in \moduli(\affine P^N)).\]
Then $|\dist(g,h)-d(g(0),h(0))|\leq (1/9)\sup_{z\in \affine}d(g(z),h(z))$.
Hence for $\Omega\subset \affine$
\begin{equation} \label{eq: dist_Omega - d(,)}
 |\dist_\Omega(g,h)-\sup_{z\in \Omega}d(g(z),h(z))|\leq \frac{1}{9}\sup_{z\in \affine}d(g(z),h(z)).
\end{equation}
Let $\delta>0$ be the positive constant introduced in Proposition \ref{prop: result of deformation}.
Let $\Lambda\subset \affine$ be an $R$-square ($R>2$).
By Proposition \ref{prop: structure of H_f}, there exists $V = V_\Lambda\subset H_f$ satisfying 
the conditions (i) and (ii) in Proposition \ref{prop: structure of H_f}.
We investigate the map $B_\delta(H_f)\to \moduli(\affine P^N)$, $u\mapsto f_u$,
(given by Proposition \ref{prop: result of deformation}) and 
its restriction to $B_\delta(V) :=V\cap B_\delta(H_f)$.

From the condition (ii) of Proposition \ref{prop: result of deformation}, for $u,v\in B_\delta(H_f)$, 
we have $\sup_{z\in \affine}d(f_u(z),f_v(z)) \leq (9/8)\norm{u-v}_{L^\infty(\affine)}$.
Hence $(B_\delta(H_f), \norm{\cdot}_{L^\infty(\affine)})\to \moduli(\affine P^N)$ is continuous.
For $u,v\in B_\delta(H_f)$
\begin{equation*}
  \begin{split}
    &\left|\dist_\Lambda(f_u,f_v)-\sup_{z\in \Lambda}|u(z)-v(z)| \right|  \\
    &\leq 
    \left|\dist_\Lambda(f_u,f_v)-\sup_{z\in \Lambda} d(f_u(z),f_v(z))\right|
    + \left|\sup_{z\in \Lambda} d(f_u(z),f_v(z))-\sup_{z\in \Lambda}|u(z)-v(z)|\right| \\
    &\leq \frac{1}{9}\sup_{z\in \affine} d(f_u(z),f_v(z)) + \frac{1}{8}\norm{u-v}_{L^\infty(\affine)}
    \quad (\text{by Proposition \ref{prop: result of deformation} (ii) and (\ref{eq: dist_Omega - d(,)})})\\
    &\leq \frac{1}{4}\norm{u-v}_{L^\infty(\affine)}.
  \end{split}
\end{equation*}
Thus 
\[ \norm{u-v}_{L^\infty(\Lambda)}\leq \dist_\Lambda(f_u,f_v)+\frac{1}{4}\norm{u-v}_{L^\infty(\affine)}.\]
For $u,v\in B_\delta(V) = V\cap B_\delta(H_f)$, we have 
$\norm{u-v}_{L^\infty(\affine)} \leq 2\norm{u-v}_{L^\infty(\Lambda)}$
(Proposition \ref{prop: structure of H_f} (ii)).
Hence 
\[ \norm{u-v}_{L^\infty(\affine)} \leq 4\, \dist_\Lambda(f_u,f_v), \quad (u,v\in B_\delta(V)).\]
Hence for $\varepsilon<\delta/4$, 
\begin{equation*}
   \begin{split}
     \widim_\varepsilon (\moduli(\affine P^N), \dist_\Lambda) &\geq 
      \widim_{4\varepsilon}(B_\delta(V), \norm{\cdot}_{L^\infty(\affine)}) \\
     & = \dim_{\mathbb{R}} V \quad (\text{by Example \ref{example: widim of Banach ball}}) \\
     & \geq 2(N+1)\int_{\Lambda}|df|^2dxdy - 2C_f R \quad 
    (\text{by Proposition \ref{prop: structure of H_f} (i)}).
   \end{split}
\end{equation*}
Since $\widim_\varepsilon(\moduli(\affine P^N), \dist_\Lambda) 
= \widim_\varepsilon(\moduli(\affine P^N), \dist_{[0,R]\times [0,R]})$, for $\varepsilon<\delta/4$,
the quantity $\widim_\varepsilon(\moduli(\affine P^N), \dist_{[0,R]\times [0,R]})$ is bounded from below by 
\[ 2(N+1)\left(\sup_{\Lambda}\int_\Lambda |df|^2dxdy\right) -2C_f R.\]
Here $\Lambda$ runs over all $R$-squares.
Dividing this by $R^2$ and letting $R\to \infty$, we get 
\[ \lim_{R\to \infty}\left(\frac{1}{R^2}\widim_\varepsilon(\moduli(\affine P^N),\dist_{[0,R]\times [0,R]})\right)
    \geq 2(N+1)\rho(f).\]
Here we have used (\ref{eq: several definition of energy density}).
Let $\varepsilon \to 0$. Then $\dim(\moduli(\affine P^N):\affine)\geq 2(N+1)\rho(f)$
by (\ref{eq: several definition of mean dimension}).
\end{proof}
\begin{remark}
The above argument also gives the lower bound on the local mean dimension 
$\dim_f(\moduli(\affine P^N):\affine)$.
(Local mean dimension is a notion introduced in \cite{Matsuo-Tsukamoto}.)
The readers can skip this remark.

Let $f:\affine \to \affine P^N$ be a non-degenerate Brody curve with $\norm{df}_{L^\infty(\affine)}<1$.
Let $B_r(f)_\affine \subset \moduli(\affine P^N)$ $(r>0)$ be the set of $g\in \moduli(\affine P^N)$ satisfying 
$\dist_\affine(f,g)\leq r$.
Since $f_0=f$, if $(4/5)r\leq \delta$ then $u\in B_{(4/5)r}(H_f)$ satisfies 
$f_u\in B_r(f)_\affine$.
Let $\Lambda\subset \affine$ be an $R$-square $(R>2)$.
As in the above proof, for $4\varepsilon < (4/5)r\leq \delta$, we get 
\[ \widim_\varepsilon(B_r(f)_\affine,\dist_\Lambda)\geq 2(N+1)\int_\Lambda |df|^2dxdy -2C_f R.\]
Hence 
\begin{equation*}
  \begin{split}
   \dim_f(\moduli(\affine P^N):\affine) &:= \lim_{r\to +0}\left\{\lim_{\varepsilon\to +0}
  \left(\lim_{R\to \infty}\frac{1}{R^2}
  \sup_{\Lambda: \text{$R$-square}}\widim_\varepsilon(B_r(f)_\affine, \dist_\Lambda)\right)\right\} \\
   &\geq 2(N+1)\rho(f).
  \end{split}
\end{equation*}
Then $\dim_{loc}(\moduli(\affine P^N):\affine) := \sup_{f\in \moduli(\affine P^N)}\dim_f(\moduli(\affine P^N):\affine)$
satisfies 
\[ 2(N+1)\rho(\affine P^N)\leq \dim_{loc}(\moduli(\affine P^N):\affine) \leq 
   \dim(\moduli(\affine P^N):\affine)  \leq 4N\rho_{\mathrm{NSA}}(\affine P^N).\]
The proof is the same as the proof of Theorem \ref{main theorem}.
In particular we get 
\[ \dim_{loc}(\moduli(\affine P^1):\affine) = \dim(\moduli(\affine P^1):\affine).\]
\end{remark}

\section{Proof of Proposition \ref{prop: result of deformation}} \label{section: deformation theory}
In this section we prove Proposition \ref{prop: result of deformation}.

\subsection{Analytic preliminaries} \label{subsection: analytic preliminaries}
Let $f:\affine\to \affine P^N$ be a Brody curve.
As in Section \ref{section: proof of Theorem mean dimension around non-degenerate curves},
let $T\affine P^N$ be the tangent bundle of $\affine P^N$ with the natural 
holomorphic vector bundle structure, 
and let $E:=f^*T\affine P^N$ be the pull-back of $T\affine P^N$.
$E$ is a holomorphic vector bundle over the complex plane $\affine$.
Its Hermitian metric $h$ is given by the pull-back of the Fubini-Study metric.
$E$ is equipped with the unitary connection $\nabla$ defined by the holomorphic structure and 
the metric $h$.

Let $1<p<\infty$ be a real number, and $k\geq 0$ be an integer.
Let $a\in L^p_{k,loc}(\Lambda^{0,i}(E))$ $(i=0,1)$ be a locally $L^p_k$-section of 
$\Lambda^{0,i}(E)$ (the $\mathcal{C}^\infty$-vector bundle of $(0,i)$-forms valued in $E$).
For a subset $\Omega\subset \affine$, we set
\[ \norm{a}_{L^p_k(\Omega)} := \left(\sum_{n=0}^k \int_\Omega |\nabla^n a|^p dxdy \right)^{1/p}.\]
We define the $\ell^{\infty}L^p_k$-norm $\norm{a}_{\ell^\infty L^p_k}$ by 
\[ \norm{a}_{\ell^\infty L^p_k} := \sup_{z\in \affine} \norm{a}_{L^p_k(D_1(z))}. \]
Let $\ell^\infty L^p_k(\Lambda^{0,i}(E))$ be the Banach space of all $a\in L^p_{k,loc}(\Lambda^{0,i}(E))$
satisfying $\norm{a}_{\ell^\infty L^p_k} <+\infty$.
\begin{lemma}  \label{lemma: Sobolev embedding}
(i) 
For $a\in L^2_{2,loc}(\Lambda^{0,i}(E))$,
\[ \norm{a}_{L^\infty(\affine)} \leq \const  \norm{a}_{\ell^\infty L^2_2}.\]
(Precisely speaking, if the right-hand-side is finite then the left-hand-side is also finite 
and satisfies the inequality.)

\noindent 
(ii)
If $a\in L^p_{2,loc}(\Lambda^{0,i}(E))$ with $p>2$, then 
\[ \norm{a}_{L^\infty(\affine)} +\norm{\nabla a}_{L^\infty(\affine)} \leq 
   \const_p \norm{a}_{\ell^\infty L^p_2}.\]
\end{lemma}
\begin{proof}
Since $\moduli(\affine P^N)$ is compact, there are $\delta>0$ and $\const_k>0$ $(k\geq 0)$ such that 
for every $z\in \affine$ there is a trivialization $u$ of the holomorphic vector bundle $E$
over a neighborhood of $D_\delta(z)$ such that 
$u_* h =(h_{\alpha\bar{\beta}})_{\alpha\beta}$ (the Hermitian matrix representing $h$ under the trivialization $u$)
satisfies 
$\norm{h_{\alpha\bar{\beta}}}_{\mathcal{C}^k(D_\delta(z))}, \norm{h^{\alpha\bar{\beta}}}_{\mathcal{C}^k(D_\delta(z))}
\leq \const_k$.
(Here $(h^{\alpha\bar{\beta}})= (h_{\alpha\bar{\beta}})^{-1}$.)
Then the norms $\norm{a}_{L^p_k(D_\delta(z))}$ and $\norm{a}_{L^\infty(D_\delta(z))}$ 
are equivalent to $\norm{u\circ a}_{L^p_k(D_\delta(z))}$ and $\norm{u\circ a}_{L^\infty(D_\delta(z))}$
uniformly in $z\in \affine$ respectively.
(We consider $u\circ a$ as a $\affine^N$-valued $(0,i)$-form in $D_\delta(z)$.)
Hence the Sobolev embedding theorem (Gilbarg-Trudinger \cite[Chapter 7.7]{Gilbarg-Trudinger}) implies 
\[ \norm{a}_{L^\infty(D_\delta(z))}\leq \const \norm{a}_{L^2_2(D_\delta(z))}.\]
Here the important point is that $\const$ is independent of $z\in \affine$.
Thus $\norm{a}_{L^\infty(\affine)}\leq \const\norm{a}_{\ell^\infty L^2_2}$.
(ii) can be proved in the same way.
\end{proof}
Let $\varphi: \affine \to \mathbb{R}$ be a $\mathcal{C}^\infty$-function satisfying 
$\norm{\varphi}_{\mathcal{C}^k(\affine)} < +\infty$ for all $k\geq 0$.
We set $\bar{\partial}^{*}_\varphi(a) := e^{-\varphi}\bar{\partial}^*(e^{\varphi}a)$ for 
$a\in \Omega^{0,1}(E)$. Here $\bar{\partial}^*$ is the formal adjoint of the 
Dolbeault operator $\bar{\partial}:\Omega^0(E) \to \Omega^{0,1}(E)$
with respect to the Hermitian metric $h$.
$\bar{\partial}^*_\varphi$ is the formal adjoint of $\bar{\partial}$ with respect to the metric 
$e^{\varphi}h$.
We define the operator $\square_\varphi: \Omega^{0,i}(E)\to \Omega^{0,i}(E)$ by setting 
\[ \square_\varphi a := \bar{\partial}^*_\varphi \bar{\partial}a \quad (i=0) ,\quad 
    \square_\varphi a := \bar{\partial} \bar{\partial}^*_\varphi a\quad (i=1). \]
\begin{lemma}  \label{lemma: L^p-estimate}
For $a\in \ell^\infty L^p_{k+2}(\Lambda^{0,i}(E))$,
\[ \norm{a}_{\ell^\infty L^p_{k+2}}\leq \const_{p,k,\varphi} \left(\norm{a}_{\ell^\infty L^p} 
    +\norm{\square_\varphi a}_{\ell^\infty L^p_k}\right).\]
More precisely, if $a\in L^p_{k+2,loc}(\Lambda^{0,1}(E))$
and the right hand side of the above is finite then $a\in \ell^\infty L^p_{k+2}$
and satisfies the above inequality.
\end{lemma}
\begin{proof}
We use the trivialization $u$ of $E$ introduced in the proof of Lemma \ref{lemma: Sobolev embedding}.
Since $\norm{\varphi}_{\mathcal{C}^l(\affine)} < +\infty$ for all $l\geq 0$,
under the trivialization $u$,
the operator $\square_\varphi$ is represented as
\[ \square_\varphi = (-1/2)\Delta + A\frac{\partial}{\partial x} + B\frac{\partial}{\partial y} +C \]
over a neighborhood of $D_\delta(z)$ where the $\mathcal{C}^l$-norms $(l\geq 0)$
of the matrices $A$, $B$, $C$ over $D_\delta(z)$ are bounded uniformly in $z\in \affine$.
Then from the $L^p$-estimate (Gilbarg-Trudinger \cite[Chapter 9.5]{Gilbarg-Trudinger})
\[ \norm{a}_{L^p_{k+2}(D_{\delta/2}(z))} \leq \const_{p,k,\varphi} 
   \left(\norm{a}_{L^p(D_\delta(z))} + \norm{\square_\varphi a}_{L^p_k(D_\delta(z))}\right).\]
The desired estimate follows from this.
\end{proof}

\subsection{Perturbation of the Hermitian metric}  \label{subsection: perturbation of the Hermitian metric}
Here we develop a perturbation technique of a Hermitian metric
(Lemma \ref{lemma: perturbation technique} below).
Gromov also discussed it in \cite[p. 399]{Gromov}.
Tsukamoto \cite[Section 4.3]{Tsukamoto-ETDS} studied an easier situation.
\begin{lemma}  \label{lemma: technical lemma for the metric perturbation}
Let $g:\affine\to \mathbb{R}_{\geq 0}$ be a non-negative smooth function with 
$\norm{g}_{\mathcal{C}^k(\affine)} < +\infty$ for all $k\geq 0$.
We suppose that 
the following non-degeneracy condition holds:
There exist $\delta>0$ and $R>0$ such that for all $p\in \affine$ we have $\norm{g}_{L^\infty(D_R(p))}
\geq \delta$.
Then there exists a smooth function $\varphi:\affine \to \mathbb{R}$ satisfying 
\[ (-\Delta+1)\varphi=-g, \quad \norm{\varphi}_{\mathcal{C}^k(\affine)} <+\infty \quad (\forall k\geq 0),\quad 
   \sup_{z\in \affine} \varphi(z) <0.\]
Here $\Delta = \partial^2/\partial x^2 + \partial^2/\partial y^2$.
\end{lemma}
\begin{proof}
We need the following sublemma.
\begin{sublemma}\label{sublemma: L^infty-estimate, baby}
Let $\varphi:\affine\to \mathbb{R}$ be a function of class $\mathcal{C}^2$ $(\varphi\in \mathcal{C}^2_{loc})$.
Suppose that the norms $\norm{\varphi}_{L^\infty(\affine)}$ and 
$\norm{(-\Delta+1)\varphi}_{L^\infty(\affine)}$ are both finite. 
Then 
\[ \norm{\varphi}_{L^\infty(\affine)} \leq 4\norm{(-\Delta+1)\varphi}_{L^\infty(\affine)}.\]
\end{sublemma}
\begin{proof}
Take $z_0\in \affine$ such that $|\varphi(z_0)|\geq \norm{\varphi}_{L^\infty(\affine)}/2$.
For simplicity, we suppose $z_0=0$.
Moreover we suppose $\varphi(0)\geq 0$. 
(If $\varphi(0)<0$ then we apply the following argument to $-\varphi$.)
We define $w:\affine\to \mathbb{R}$ by 
\[ w(z) := \frac{1}{2\pi}\int_0^{2\pi}e^{(x\cos \theta + y\sin\theta)/{\sqrt{2}}}d\theta.\]
$w$ satisfies 
\[ (-\Delta + 1/2)w=0, \quad \min_{z\in \affine}w(z) = w(0) = 1, \quad 
    w(z)\to +\infty \quad (|z|\to +\infty).\]
Then $(-\Delta+1)w = w/2 \geq 1/2$.
For $\varepsilon>0$, set $M:=2\norm{(-\Delta+1)\varphi}_{L^\infty(\affine)}+\varepsilon>0$.
\[ (-\Delta+1)\left(M w -\varphi\right) \geq M/2 - (-\Delta+1)\varphi
    \geq \norm{(-\Delta+1)\varphi}_{L^\infty(\affine)} +\varepsilon/2 -(-\Delta+1)\varphi
   \geq \varepsilon/2.\]
Since the function $Mw -\varphi$ is positive for $|z|\gg 1$,
the weak minimum principle 
(Gilbarg-Trudinger \cite[Chapter 3.1, Corollary 3.2]{Gilbarg-Trudinger}) implies 
that this function is non-negative everywhere. Hence 
\[ \norm{\varphi}_{L^\infty(\affine)}/2 \leq \varphi(0) \leq Mw(0) = M 
   = 2\norm{(-\Delta+1)\varphi}_{L^\infty(\affine)}+\varepsilon.\]
Let $\varepsilon\to 0$. We get 
\[ \norm{\varphi}_{L^\infty(\affine)}\leq 4\norm{(-\Delta+1)\varphi}_{L^\infty(\affine)}.\]
\end{proof}
Let $\phi_n:\affine\to [0,1]$ $(n\geq 1)$ be a cut-off function such that $\phi_n=1$ over 
$D_n(0)$ and $\supp(\phi_n)\subset D_{n+1}(0)$.
We want to solve the equation $(-\Delta+1)\varphi=-\phi_n g$.
The following is a standard $L^2$-argument.

Let $L^2_1(\affine)$ be the space of $L^2$-functions $\varphi:\affine \to \mathbb{R}$
satisfying $\partial \varphi/\partial x, \partial \varphi/\partial y\in L^2$
with the inner product 
$\langle \varphi, \varphi'\rangle_{L^2_1} := \langle \varphi, \varphi'\rangle_{L^2}
+\langle \partial \varphi/\partial x, \partial \varphi'/\partial x\rangle_{L^2} + 
\langle \partial \varphi/\partial y, \partial \varphi'/\partial y\rangle_{L^2}$.
Consider the bounded linear functional:
\[ L^2_1(\affine)\to \mathbb{R}, \quad \varphi \mapsto -\langle \varphi,\phi_n g\rangle_{L^2}.\]
From the Riesz representation theorem, there uniquely exists $\varphi_n\in L^2_1(\affine)$ satisfying 
$\langle \varphi, \varphi_n\rangle_{L^2_1} = -\langle \varphi, \phi_n g\rangle_{L^2}$
for all $\varphi\in L^2_1(\affine)$.
This implies $(-\Delta +1)\varphi_n = -\phi_n g$ as a distribution.
From the local elliptic regularity, $\varphi_n$ is smooth and $\norm{\varphi_n}_{L^\infty(\affine)} <+\infty$.
Then we can apply Sublemma \ref{sublemma: L^infty-estimate, baby} to $\varphi_n$ and get 
\[ \norm{\varphi_n}_{L^\infty(\affine)}\leq 4\norm{\phi_n g}_{L^\infty(\affine)} \leq 
   4\norm{g}_{L^\infty(\affine)} < +\infty.\]
By the local elliptic regularity, for every compact subset $K\subset \affine$ and $k\geq 0$, the sequence
$\norm{\varphi_n}_{\mathcal{C}^k(K)}$ $(n\geq 1)$ is bounded.
Then we can choose a subsequence $n_1<n_2<n_3<\dots$ such that 
$\varphi_{n_k}$ converges to some $\varphi$ in $\mathcal{C}^\infty$ over every compact subset of $\affine$.
$\varphi$ satisfies $(-\Delta+1)\varphi=-g$ and $\norm{\varphi}_{L^\infty(\affine)}\leq 4\norm{g}_{L^\infty(\affine)}$.
By the elliptic regularity, $\norm{\varphi}_{\mathcal{C}^k(\affine)} < +\infty$ for all $k\geq 0$.

Note that we have not used the non-degeneracy condition of the function $g$ so far.
We need it for the proof of the condition $\sup_{z\in \affine}\varphi(z)<0$.

Set $M:=\sup_{z\in \affine}\varphi(z)$.
There are $z_n\in \affine$ $(n\geq 1)$ such that $\varphi(z_n) \to M$.
Set $\varphi_n (z) := \varphi(z+z_n)$ and $g_n(z) := g(z+z_n)$.
Then 
\[ (-\Delta+1)\varphi_n = -g_n.\]
The sequences $\norm{\varphi_n}_{\mathcal{C}^k(\affine)}$ and $\norm{g_n}_{\mathcal{C}^k(\affine)}$
$(n\geq 1)$ are bounded for every $k\geq 0$.
Hence by choosing a subsequence (denoted also by $\varphi_n$ and $g_n$), we can assume that 
$\varphi_n$ and $g_n$ converge to $\varphi_\infty$ and $g_\infty$ respectively in $\mathcal{C}^\infty$
over every compact subset of $\affine$.
They satisfy
\[ g_\infty \geq 0, \quad 
   (-\Delta+1)\varphi_\infty = -g_\infty \leq 0, \quad 
   \varphi_\infty(z) \leq \varphi_\infty(0) = M.\]
From the non-degeneracy condition of $g$, the function $g_\infty$ is not zero.
Hence if $\varphi_\infty$ is a constant, then $\varphi_\infty = -g_\infty$ is a negative constant function
and $M <0$.
If $\varphi_\infty$ is not a constant, then
the strong maximum principle \cite[Chapter 3.2, Theorem 3.5]{Gilbarg-Trudinger} implies
that $\varphi_\infty$ cannot achieve a non-negative maximum value.
Hence $M=\varphi_\infty(0) = \max_{z\in \affine} \varphi_\infty(z) <0$.
\end{proof}
Recall that $f:\affine\to \affine P^N$ is a Brody curve and $E=f^*T\affine P^N$.
For $a\in \Omega^{0,1}(E)$ we have the Weintzenb\"{o}ck formula:
\begin{equation} \label{eq: Weintzenbock formula}
 \bar{\partial}\bar{\partial}^* a = \frac{1}{2}\nabla^*\nabla a + \Theta a,
\end{equation}
where $\Theta :=[\nabla_{\partial/\partial z},\nabla_{\partial/\partial \bar{z}}]$ is the 
curvature operator.
The crucial fact for the analysis of this paper is that
the holomorphic bisectional curvature of the Fubini-Study metric is positive.
From this, there exists a positive constant $c$ such that 
\[ h(\Theta a, a) \geq c |df|^2 |a|^2.\]
This means that the curvature operator is positive where $|df|$ is positive.
The non-degeneracy condition of the map $f$ enters into the argument through this point.
(See the condition (ii) of Definition-Lemma \ref{defintion-lemma: non-degenerate}.)
In the next lemma we will prove that if $f$ is non-degenerate then we can perturb the 
Hermitian metric $h$ so that the curvature is uniformly positive:
\begin{lemma} \label{lemma: perturbation technique}
Let $f:\affine\to \affine P^N$ be a non-degenerate Brody curve.
There is a smooth function $\varphi:\affine \to \mathbb{R}$ with
$\norm{\varphi}_{\mathcal{C}^k(\affine)}<+\infty$ $(\forall k\geq 0)$
satisfying the following.
Let $\Theta_\varphi$ be the curvature of the Hermitian metric $h_\varphi:=e^\varphi h$.
Then there is $c'>0$ such that 
\[ h_{\varphi}(\Theta_\varphi a, a) \geq c' |a|_{h_\varphi}^2 \]
for all $a\in \Omega^{0,1}(E)$.
\end{lemma}
\begin{proof}
We have 
$\Theta_\varphi a= \frac{-\Delta \varphi}{4}a+\Theta a$ for $a\in \Omega^{0,1}(E)$, and hence 
\[ h_\varphi(\Theta_\varphi a, a) = e^\varphi\left(\frac{-\Delta\varphi}{4}|a|_h^2 + h(\Theta a, a)\right)
   \geq e^{\varphi}\left(\frac{-\Delta\varphi}{4} + c|df|^2\right)|a|_h^2.\]
By the non-degeneracy of $f$ and Lemma \ref{lemma: technical lemma for the metric perturbation},
there is a smooth function $\varphi:\affine \to \mathbb{R}$ satisfying 
\[ (-\Delta+1)\varphi = -4c|df|^2,\quad \norm{\varphi}_{\mathcal{C}^k(\affine)}<+\infty\quad 
(\forall k\geq 0), \quad 
    \sup_{z\in \affine}\varphi(z) <0. \]
Then 
\[ h_\varphi(\Theta_\varphi a, a) \geq e^{\varphi}(-\varphi/4)|a|_h^2 = 
  (-\varphi/4)|a|_{h_\varphi}^2 \geq (-\sup_{z\in \affine}\varphi(z)/4)|a|_{h_\varphi}^2.\]
Hence $c':=-\sup_{z\in \affine}\varphi(z)/4>0$ satisfies the statement.
\end{proof}
In our convention, the Fubini-Study metric $g_{i\bar{j}}$ on $\affine P^N$ is given by 
\[ g_{i\bar{j}} = \frac{1}{2\pi}\frac{\partial^2}{\partial z_i \partial \bar{z}_j}\log(1+|z_1|^2+\dots+|z_N|^2) \]
over $\{[1:z_1:\dots:z_N]\}\subset \affine P^N$. 
The spherical derivative $|df|(z)$ for a holomorphic curve $f:\affine \to \affine P^N$ satisfies  
\begin{equation}  \label{eq: spherical derivative is dilatation}
 f^*\left(\sqrt{-1}\sum g_{i\bar{j}}dz_i d\bar{z}_j\right) = |df|^2 dxdy.
\end{equation}

The Fubini-Study metric $g_{i\bar{j}}$ satisfies the K\"{a}hler-Einstein equation
\[ \mathrm{Ric}_{i\bar{j}} = -\frac{\partial^2}{\partial z_i \partial\bar{z}_j} \log (\det(g_{k\bar{l}}))
   = 2\pi(N+1)g_{i\bar{j}}.\]
From this, the curvature operator
$\Theta = [\nabla_{\partial/\partial z}, \nabla_{\partial/\partial\bar{z}}]$ 
in (\ref{eq: Weintzenbock formula}) satisfies 
\begin{equation} \label{eq: Chern form}
 \frac{\sqrt{-1}}{2\pi} \mathrm{tr}(\Theta)dzd\bar{z} = (N+1)|df|^2 dxdy 
\end{equation}
since $\mathrm{tr}(\Theta)dz d\bar{z} = f^{*}(\sum\mathrm{Ric}_{i\bar{j}}dz_i d\bar{z}_j)$.
The equation (\ref{eq: Chern form}) will be used 
in the proof of Proposition \ref{prop: preliminary of prop: structure of H_f}.
Note that the form $(\sqrt{-1}/2\pi)\mathrm{tr}(\Theta)dzd\bar{z}$ is the Chern form
representing $c_1(E)$ although we have $c_1(E)=0$ because $H^2(\affine;\mathbb{Z})=0$.

\subsection{$L^\infty$-estimate} \label{subsection: L^infty-estimate}

Let $f:\affine\to \affine P^N$ be a non-degenerate Brody curve, and 
let $\varphi:\affine \to \mathbb{R}$ be a smooth function introduced in 
Lemma \ref{lemma: perturbation technique}.
Propositions \ref{prop: L^infty-estimate} and \ref{prop: existence of inverse of square}
below essentially use the positivity of the curvature $\Theta_\varphi$.

The following $L^\infty$-estimate was proved in \cite[Proposition 4.2]{Tsukamoto-ETDS}.
\begin{proposition} \label{prop: L^infty-estimate}
Let $a\in \Omega^{0,1}(E)$ be an $E$-valued $(0,1)$-form of class $\mathcal{C}^2$ ($a\in \mathcal{C}^2_{loc}$).
Set $b:=\square_\varphi a$.
If $\norm{a}_{L^\infty(\affine)}, \norm{b}_{L^\infty(\affine)} < +\infty$, then 
\[ \norm{a}_{L^\infty(\affine)}\leq \const_{f,\varphi} \norm{b}_{L^\infty(\affine)}.\]
\end{proposition}
\begin{proof}
The proof is similar to the proof of Sublemma \ref{sublemma: L^infty-estimate, baby}.
For the detail, see \cite[pp. 1648-1649]{Tsukamoto-ETDS}.
\end{proof}
\begin{proposition} \label{prop: existence of inverse of square}
Let $b\in L^2_{2,loc}(\Lambda^{0,1}(E))$ and suppose $\norm{b}_{L^\infty(\affine)} < +\infty$.
Then there uniquely exists $a\in L^2_{4,loc} (\Lambda^{0,1}(E))$ satisfying 
\[ \square_\varphi a = b, \quad \norm{a}_{L^\infty(\affine)} < +\infty.\]
Moreover $\norm{a}_{L^\infty(\affine)}+\norm{\nabla a}_{L^\infty(\affine)} 
\leq \const_{f,\varphi}\norm{b}_{L^\infty(\affine)}$.
\end{proposition}
\begin{proof}
The uniqueness follows from Proposition \ref{prop: L^infty-estimate}.
(Note the Sobolev embedding $L^2_{4,loc}\hookrightarrow \mathcal{C}^2_{loc}$ in $\mathbb{R}^2$.)
So the problem is the existence.
We have the Weinzenb\"{o}ck formula: for $a\in \Omega^{0,1}(E)$
\[ \square_\varphi a = \frac{1}{2}\nabla_\varphi^*\nabla_\varphi a + \Theta_\varphi a,\]
where $\nabla_\varphi$ is the unitary connection on $E$ with respect to the metric $h_\varphi = e^\varphi h$.
$\Theta_\varphi$ satisfies the positivity condition in Lemma \ref{lemma: perturbation technique}.

Let $\phi_n :\affine \to [0,1]$ be a cut-off function such that $\phi_n = 1$ over $D_n(0)$ and
$\supp(\phi_n)\subset D_{n+1}(0)$.
From the positivity of the curvature, 
as in the proof of Lemma \ref{lemma: technical lemma for the metric perturbation},
a standard $L^2$-argument shows that 
there is $a_n\in L^2_1(\Lambda^{0,1}(E))$ (the space of $L^2$-sections $a$ of $\Lambda^{0,1}(E)$ satisfying 
$\nabla_\varphi a\in L^2$) satisfying $\square_\varphi a_n =\phi_n b$ as a distribution.
(For the detail, see \cite[Lemma 5.3]{Tsukamoto-ETDS}.)
The local elliptic regularity implies $a_n\in L^2_{4,loc}$.
By Lemmas \ref{lemma: Sobolev embedding} (i) and \ref{lemma: L^p-estimate},
\begin{equation*}
 \begin{split}
  \norm{a_n}_{L^\infty(\affine)}&\leq \const\norm{a_n}_{\ell^\infty L^2_2}\leq \const_\varphi
   \left(\norm{a_n}_{\ell^\infty L^2} + \norm{\square_\varphi a_n}_{\ell^\infty L^2}\right) \\
   &\leq \const_\varphi \left(\norm{a_n}_{L^2} + \norm{\phi_n b}_{L^\infty(\affine)}\right) <+\infty.
 \end{split}
\end{equation*}
By Proposition \ref{prop: L^infty-estimate} we have 
$\norm{a_n}_{L^\infty(\affine)}\leq \const_{f,\varphi} \norm{b_n}_{L^\infty(\affine)}
\leq \const_{f,\varphi} \norm{b}_{L^\infty(\affine)}$.
Then for any compact set $K\subset \affine$
the sequence $\norm{a_n}_{L^2_2(K)}$ $(n\geq 1)$ is bounded.
By choosing a subsequence $n_1<n_2<n_3<\dots$, the sequence $a_{n_k}$ converges to some $a$ weakly 
in $L^2_2(D_R(0))$ (and hence strongly in $L^\infty(D_R(0))$) for every $R>0$.
$a$ satisfies $\square_\varphi a=b$, and 
$\norm{a}_{L^\infty(\affine)} \leq \sup_{n\geq 1}\norm{a_n}_{L^\infty(\affine)} 
\leq \const_{f,\varphi}\norm{b}_{L^\infty(\affine)}$.
By the local elliptic regularity $a\in L^2_{4,loc}$.
By Lemmas \ref{lemma: Sobolev embedding} (ii) and \ref{lemma: L^p-estimate}
\[ \norm{a}_{L^\infty(\affine)} + \norm{\nabla a}_{L^\infty(\affine)} \leq \const \norm{a}_{\ell^\infty L^3_2}
   \leq \const_\varphi\left(\norm{a}_{\ell^\infty L^3}+\norm{b}_{\ell^\infty L^3}\right)
   \leq \const_{f,\varphi} \norm{b}_{L^\infty(\affine)}.\]
\end{proof}

\subsection{Deformation theory} \label{subsection: deformation}

Let $f:\affine \to \affine P^N$ be a non-degenerate Brody curve with $\norm{df}_{L^\infty(\affine)} <1$.
In this subsection we study a deformation of $f$ and 
prove Proposition \ref{prop: result of deformation}.
Gromov \cite[pp. 399-400, Projective interpolation theorem]{Gromov} 
studied a different kind of deformation theory.
Our argument is a generalization of the deformation theory of elliptic Brody curves
developed in \cite{Tsukamoto-ETDS}.

Consider the following map (see McDuff-Salamon \cite[p. 40]{Mcduff-Salamon}):
\[ \Phi: \ell^\infty L^2_3(E)\to \ell^\infty L^2_2(\Lambda^{0,1}(E)), \quad 
    u\mapsto P_u(\bar{\partial} \exp u) \otimes d\bar{z}.\]
Here $\exp u = \exp_{f(z)} u(z)$ is defined by the exponential map of the Fubini-Study metric, and 
\[ \bar{\partial} \exp u := \frac{1}{2}
   \left(\frac{\partial}{\partial x}\exp u + J \frac{\partial}{\partial y}\exp u\right)\quad 
 (\text{$J$: complex structure of $\affine P^N$}). \]
$P_{u(z)} :T_{\exp_{f(z)} u(z)} \affine P^N\to T_{f(z)} \affine P^N$ is the parallel translation along 
the geodesic $\exp_{f(z)} (t u(z))$ $(0\leq t\leq 1)$.

$\Phi$ is a smooth map between the Banach spaces. 
$\Phi(0)=0$ and the derivative of $\Phi$ at the origin is 
equal to the Dolbeault operator:
\[ d\Phi_0 = \bar{\partial} :\ell^\infty L^2_3(E) \to \ell^\infty L^2_2(\Lambda^{0,1}(E)).\]
\begin{proposition}  \label{prop: small deformation is unobstructed}
There is a bounded linear operator $Q: \ell^\infty L^2_2(\Lambda^{0,1}(E))\to \ell^\infty L^2_3(E)$
satisfying $\bar{\partial} \circ Q=1$.
\end{proposition}
\begin{proof}
We will prove that the map 
\begin{equation} \label{eq: square is isomorphism}
 \square_\varphi = \bar{\partial}\bar{\partial}^*_\varphi:
 \ell^\infty L^2_4(\Lambda^{0,1}(E))\to \ell^\infty L^2_2(\Lambda^{0,1}(E))
\end{equation}
is an isomorphism.
($\varphi:\affine \to \mathbb{R}$ is a smooth function introduced in 
Lemma \ref{lemma: perturbation technique}.)
Then $Q:=\bar{\partial}^*_\varphi \square_\varphi^{-1}:\ell^\infty L^2_2(\Lambda^{0,1}(E))\to 
\ell^\infty L^2_3(\Lambda^{0,1}(E))$ becomes a right inverse of $\bar{\partial}$.
The injectivity of the map (\ref{eq: square is isomorphism}) directly follows from the 
$L^\infty$-estimate in Proposition \ref{prop: L^infty-estimate}.

On the other hand, 
by Proposition \ref{prop: existence of inverse of square},
for every $b\in \ell^\infty L^2_2(\Lambda^{0,1}(E))$ there 
is $a\in L^\infty \cap L^2_{4,loc}(\Lambda^{0,1}(E))$ satisfying $\square_\varphi a=b$.
By Lemma \ref{lemma: L^p-estimate}, $a\in \ell^\infty L^2_4$.
Thus the map (\ref{eq: square is isomorphism}) is surjective.
\end{proof}
Let $H_f$ be the Banach space of all $L^\infty$-holomorphic sections of $E$ introduced in 
Section \ref{section: proof of Theorem mean dimension around non-degenerate curves}.
$H_f$ is equal to the kernel of the map 
$\bar{\partial} :\ell^\infty L^2_3(E)\to \ell^\infty L^2_2(\Lambda^{0,1}(E))$
by Lemmas \ref{lemma: Sobolev embedding} and \ref{lemma: L^p-estimate}.
Moreover the norms $\norm{\cdot}_{\ell^\infty L^2_k}$ $(k\geq 0)$ are all equivalent to 
the norm $\norm{\cdot}_{L^\infty(\affine)}$ over $H_f$.

From Proposition \ref{prop: small deformation is unobstructed} and the implicit function theorem, 
there are $r>0$ and a smooth map 
$\alpha: \{u\in H_f| \norm{u}_{L^\infty(\affine)} <r\} \to \mathrm{Im} Q$ 
($\mathrm{Im} Q \subset \ell^\infty L^2_3(E)$ is a closed subspace) such that 
\[ \Phi(u+\alpha(u))=0, \quad \alpha(0)=0, \quad d\alpha_0=0.\]
The first and second conditions imply 
that $f_u :=\exp_f(u+\alpha(u))$ becomes a holomorphic curve with $f_0=f$.
The third condition implies that for any $\varepsilon>0$ there exists $0<\delta<r$ such that 
if $u,v\in H_f$ satisfies $\norm{u}_{L^\infty(\affine)}, \norm{v}_{L^\infty(\affine)} \leq \delta$ then 
$\norm{\alpha(u) -\alpha(v)}_{L^\infty(\affine)} \leq \varepsilon \norm{u-v}_{L^\infty(\affine)}$.
\begin{proof}[Proof of Proposition \ref{prop: result of deformation}]
Since $\norm{df}_{L^\infty(\affine)} < 1$, if $\delta\ll 1$, the holomorphic curves $f_u$ $(u\in B_\delta(H_f))$
satisfy $\norm{df_u}_{L^\infty(\affine)}\leq 1$.
We will prove that if $0<\delta <r$ is sufficiently small then 
the map
\begin{equation*} 
   B_\delta(H_f)\ni u\mapsto f_u \in \moduli(\affine P^N) 
\end{equation*}
satisfies the conditions in Proposition \ref{prop: result of deformation}.
The condition (i) ($f_0=f$) is OK. 
So we want to prove the condition (ii).

We choose $0<\delta<r$ sufficiently small so that 
all $u,v\in B_\delta(H_f)$ satisfy 
\[ \norm{\alpha(u)-\alpha(v)}_{L^\infty(\affine)} \leq (1/20)\norm{u-v}_{L^\infty(\affine)},\]
and that if $v_1,v_2\in T_p\affine P^N$ are two tangent vectors satisfying $|v_1|,|v_2|\leq 2\delta$
then 
\[ |d(\exp(v_1),\exp(v_2))-|v_1-v_2||\leq (1/20)|v_1-v_2|.\]
The former condition comes from $d\alpha_0=0$, and the latter is just a standard property of the 
exponential map.
Then all $u,v\in B_\delta(H_f)$ satisfy
\begin{equation*}
  \begin{split}
   \left|d(\exp(u+\alpha(u)),\exp(v+\alpha(v)))-|u+\alpha(u)-v-\alpha(v)|\right|
   \leq (1/20) \left|u+\alpha(u)-v-\alpha(v)\right| \\
   \leq (1/20) \norm{u-v}_{L^\infty(\affine)} +(1/20)\norm{\alpha(u)-\alpha(v)}_{L^\infty(\affine)}
   \leq (1/20 + 1/400) \norm{u-v}_{L^\infty(\affine)}, 
  \end{split}
\end{equation*}
and 
\begin{equation*}
   \left| |u+\alpha(u)-v-\alpha(v)| - |u-v| \right|  \leq 
   \left| \alpha(u)-\alpha(v) \right| 
   \leq (1/20) \norm{u-v}_{L^\infty(\affine)}.
\end{equation*}
These inequalities imply 
the condition (ii):
\[ \left|d(\exp(u+\alpha(v)),\exp(v+\alpha(v)))-|u-v|\right| \leq 
    (1/8)\norm{u-v}_{L^\infty(\affine)}.\]
\end{proof}

\section{Study of $H_f$: proof of Proposition \ref{prop: structure of H_f}}\label{section: study of H_f}

In this section we prove Proposition \ref{prop: structure of H_f}.
Let $R>0$, and let $\Lambda =[a_1,b_1]\times [a_2,b_2]\subset \affine$ be an $R$-square
(i.e. $b_1=a_1+R$ and $b_2=a_2+R$).
For $0<r<R/2$, we set 
\[ \partial_r\Lambda = \{([a_1,a_1+r)\cup (b_1-r,b_1])\times [a_2,b_2]\} \cup
  \{[a_1,b_1]\times ([a_2,a_2+r) \cup (b_2-r,b_2])\}.\]
(This notation is used only in this section.
It conflicts with the notation $\partial_r \Omega$ introduced in Section \ref{subsection: review of mean dimension}.)
The following is a preliminary version of Proposition \ref{prop: structure of H_f}.
\begin{proposition} \label{prop: preliminary of prop: structure of H_f}
Let $f:\affine\to \affine P^N$ be a Brody curve.
Let $\varepsilon>0$, and let $\Lambda\subset \affine$ be an $R$-square with $R>2$.
Then there exists a finite dimensional complex subspace $W\subset \Omega^0(E)$ 
(the space of $\mathcal{C}^\infty$-sections of $E=f^*T\affine P^N$)
satisfying the following three conditions.

\noindent 
(i) 
\[ \dim_\affine W\geq (N+1)\int_{\Lambda}|df|^2dxdy -C_{\varepsilon} R,\]
where $C_{\varepsilon}$ is a constant depending only on $\varepsilon$.
(The important point is that it is independent of $R$.)

\noindent 
(ii) All $u\in W$ satisfy $u=0$ outside of $\Lambda$.

\noindent 
(iii) All $u\in W$ satisfy  
$\norm{\bar{\partial}u}_{L^\infty(\affine)} \leq \varepsilon \norm{u}_{L^\infty(\affine)}$.
\end{proposition}
\begin{proof}
Set $\Lambda = [a_1,b_1]\times [a_2,b_2]$.
Let $\varphi_i:\mathbb{R}\to \mathbb{R}$ $(i=1,2)$ be smooth functions such that $0\leq \varphi_i'\leq 1$, 
$\varphi_i(x)=x$ over $[a_i+1/2,b_i-1/2]$, $\varphi(x)=\varphi(a_i+1/4)$ over $x\leq a_i+1/4$
and $\varphi_i(x) = \varphi(b_i-1/4)$ over $x\geq b_i-1/4$.
Moreover we assume that, for $k\geq 1$, $|\varphi_i^{(k)}|\leq \const_k$
(depending only on $k\geq 1$).

We define a $\mathcal{C}^\infty$-map $\tilde{f}:\affine \to \affine P^N$ by 
$\tilde{f}(x+\sqrt{-1}y) := f(\varphi_1(x)+\sqrt{-1}\varphi_2(y))$.
We have $|d\tilde{f}|(z) := \max_{u\in T_z\affine,\, |u|=1}|d\tilde{f}(u)|\leq 1$ for all $z\in \affine$.
Let $\tilde{E} :=\tilde{f}^*T\affine P^N$ be the pull-back of $T\affine P^N$ by $\tilde{f}$.
$\tilde{E}$ is a complex vector bundle over $\affine$ with the 
Hermitian metric $\tilde{h}$ (the pull-back of the Fubini-Study metric) and
the unitary connection $\tilde{\nabla}$ (the pull-back of the Levi-Civita connection on $T\affine P^N$).
From the definition of $\tilde{f}$, the connection $\tilde{\nabla}$ is flat over $\partial_{1/4}\Lambda$.
Flat connections over $\partial_{1/4}\Lambda$ are classified by their holonomy maps
$\pi_1(\partial_{1/4}\Lambda)\to U(N)$.
Hence there is a bundle trivialization (as a Hermitian vector bundle) $g$ of $\tilde{E}$ over $\partial_{1/4}\Lambda$
such that $g(\tilde{\nabla})= d+A$ ($A$: connection matrix) satisfies 
\[ \norm{A}_{\mathcal{C}^k(\partial_{1/4}\Lambda)} \leq \const_k \quad (k\geq 0).\]
Here $\const_k$ are universal constants depending only on $k$.
(The important point is that they are independent of $R$.)
Let $\psi:\Lambda\to [0,1]$ be a cut-off function such that 
$\psi = 1$ over $\Lambda\setminus \partial_{1/5}\Lambda$, $\psi=0$ over $\partial_{1/6}\Lambda$,
and $\norm{\psi}_{\mathcal{C}^k(\Lambda)}\leq \const_k$.
We define a unitary connection $\nabla'$ on $\tilde{E}$ over $\Lambda$ by 
$\nabla' := g^{-1}(d+\psi A)$.
($\nabla'=\tilde{\nabla}$ over $\Lambda\setminus \partial_{1/5}\Lambda$.)
Under the trivialization $g$, the metric $\tilde{h}$ and the connection $\nabla'$ are equal to the standard metric and the 
product connection of $\partial_{1/6}\Lambda \times \affine^N$ over $\partial_{1/6}\Lambda$.

Consider an elliptic curve $\mathbb{T} := \affine/(R \mathbb{Z} + R\sqrt{-1}\mathbb{Z})$, and 
let $\pi:\affine\to \mathbb{T}$ be the natural projection.
We define a complex vector bundle $E'$ over $\mathbb{T}$ as follows.
$E' = \tilde{E}$ over 
$\pi(\Lambda\setminus \partial_{1/5}\Lambda)\cong  \Lambda\setminus \partial_{1/5}\Lambda$, 
and $E'|_{\pi(\partial_{1/4}\Lambda)}$ is 
equal to the product bundle $\pi(\partial_{1/4}\Lambda)\times \affine^N$.
We glue these by the map $g$.
The metric $\tilde{h}$ and the connection $\nabla'$ naturally descend to the metric and connection on 
$E'$ (also denoted by $\tilde{h}$ and $\nabla'$).

Let $\Theta' :=[\nabla'_{\partial/\partial z},\nabla'_{\partial/\partial\bar{z}}]$ be the curvature of $\nabla'$.
From the definition, $\Theta' = [\nabla_{\partial/\partial z},\nabla_{\partial /\partial \bar{z}}]$ over 
$\pi(\Lambda\setminus \partial_{1/2}\Lambda)\cong \Lambda\setminus \partial_{1/2}\Lambda$, and
$|\Theta'|\leq \const$ (a universal constant) all over $\mathbb{T}$.
Then by (\ref{eq: Chern form})
\begin{equation} \label{eq: c_1 of E'}
 \int_{\mathbb{T}}c_1(E') = \frac{\sqrt{-1}}{2\pi}\int_{\mathbb{T}}\mathrm{tr}(\Theta') dz d\bar{z}
 \geq (N+1)\int_{\Lambda} |df|^2 dxdy -\const \cdot R.
\end{equation}
Let $\bar{\partial}_{\nabla'}:\Omega^0(E')\to \Omega^{0,1}(E')$ be the Dolbeault operator
over $\mathbb{T}$ twisted by the unitary connection $\nabla'$
(i.e. the $(0,1)$-part of the covariant derivative $\nabla':\Omega^0(E)\to \Omega^1(E)$).
Let $H^0_{\nabla'}$ be the space of $u\in \Omega^0(E')$ satisfying $\bar{\partial}_{\nabla'}u=0$.
From the Riemann-Roch formula and the above (\ref{eq: c_1 of E'})
\begin{equation} \label{eq: riemann-roch}
  \dim_\affine H^0_{\nabla'} \geq \int_{\mathbb{T}}c_1(E') \geq (N+1)\int_{\Lambda}|df|^2 dxdy -\const\cdot R.
\end{equation}
\begin{lemma} \label{lemma: quantitative elliptic estimate}
  For all $u\in H^0_{\nabla'}$,
\[ \norm{\nabla' u}_{L^\infty(\mathbb{T})} \leq K \norm{u}_{L^\infty(\mathbb{T})}.\]
Here $K$ is a universal constant (independent of $f$, $R$, $\Lambda$).
\end{lemma}
\begin{proof}
The connection $\nabla'$ has the following property:
There is a universal constant $r>0$ such that 
for every $p\in \mathbb{T}$ there is a bundle trivialization $v$ of a Hermitian vector bundle $E'$ 
over $D_r(p)$ satisfying $v(\nabla') = d+A'$ with 
\[ \norm{A'}_{\mathcal{C}^k(D_r(p))} \leq \const_k \quad (k\geq 0). \]
Then the result follows from the elliptic regularity.
\end{proof}
Let $\tau =\tau(\varepsilon)>0$ be a small number which will be fixed later.
We take points $p_1,\dots,p_M \in \pi(\partial_1\Lambda)$ with $M\leq \const_\tau\cdot R$ such that 
for every $p\in \pi(\partial_1\Lambda)$ there is $p_i$ satisfying 
$d(p,p_i)\leq \tau$.
We define $V\subset H^0_{\nabla'}$ as the space of $u\in H^0_{\nabla'}$ satisfying 
$u(p_i)=0$ for all $i=1,\dots,M$.
From (\ref{eq: riemann-roch}),
\begin{equation}  \label{eq: riemann-roch minus the number of points}
 \dim_\affine V\geq \dim_{\affine}H^0_{\nabla'} - \dim_{\affine}\left(\bigoplus_{i=1}^M E'_{p_i}\right)
 \geq (N+1)\int_{\Lambda} |df|^2dxdy - C_{\varepsilon} R.
\end{equation}
Let $u\in V$ and $p\in \pi(\partial_1\Lambda)$.
Take $p_i$ satisfying $d(p,p_i)\leq \tau$.
From $u(p_i)=0$ and 
Lemma \ref{lemma: quantitative elliptic estimate}, 
\[ |u(p)|\leq \tau\norm{\nabla' u}_{L^\infty(\mathbb{T})} \leq \tau K \norm{u}_{L^\infty(\mathbb{T})}.\]
We choose $\tau>0$ so that $\tau K < 1$.
Then the maximum of $|u|$ is attained in $\mathbb{T}\setminus \pi(\partial_1\Lambda)$.

Let $\phi: \affine\to \mathbb{R}$ be a cut-off such that $\phi=1$ over $\Lambda\setminus \partial_1\Lambda$, 
$\supp(\phi)$ is contained in the interior of 
$\Lambda\setminus \partial_{1/2}\Lambda$, and $|d\phi|\leq 10$.
For $u\in V$, we set $u' := \phi u$.
Here we identify the region $\Lambda\setminus \partial_{1/2}\Lambda$ with $\pi(\Lambda\setminus \partial_{1/2}\Lambda)$
where we have $E'=E$, 
and we consider $u'$ as a section of $E$ over the plane $\affine$.
Set $W:=\{u'|u\in V\}$.
We have $\norm{u'}_{L^\infty(\affine)}=\norm{u}_{L^\infty(\mathbb{T})}$.
Hence, by (\ref{eq: riemann-roch minus the number of points}), we get the condition (i):
\[ \dim_\affine W = \dim_\affine V \geq (N+1)\int_{\Lambda} |df|^2dxdy - C_{\varepsilon} R.\]
The condition (ii) is obviously satisfied.
$\bar{\partial}u' = \bar{\partial}\phi\otimes u$ is supported in $\partial_1\Lambda$.
\[ \norm{\bar{\partial}u'}_{L^\infty(\affine)} \leq 10\norm{u}_{L^\infty(\pi(\partial_1\Lambda))}
   \leq 10\tau K \norm{u}_{L^\infty(\mathbb{T})} 
   = 10\tau K \norm{u'}_{L^\infty(\affine)}.\]
We choose $\tau>0$ so that $10\tau K \leq \varepsilon$.
Then the condition (iii) is satisfied.
\end{proof}
\begin{proof}[Proof of Proposition \ref{prop: structure of H_f}]
Let $\varepsilon>0$ be a small number which will be fixed later.
By Proposition \ref{prop: preliminary of prop: structure of H_f}, for this $\varepsilon$ and 
any $R$-square $\Lambda$ $(R>2)$, 
there is a finite dimensional complex subspace $W\subset \Omega^0(E)$ satisfying the 
conditions (i), (ii), (iii) in Proposition \ref{prop: preliminary of prop: structure of H_f}.
By Proposition \ref{prop: existence of inverse of square}, there is a linear map 
\[ W\to \Omega^{0,1}(E), \quad u \mapsto a ,\]
such that 
\[ \bar{\partial}\bar{\partial}^*_{\varphi}a = \bar{\partial}u, \quad   
   \norm{\bar{\partial}^*_\varphi a}_{L^\infty(\affine)} \leq C'_f \norm{\bar{\partial}u}_{L^\infty(\affine)}
   \leq C'_f \cdot \varepsilon \norm{u}_{L^\infty(\affine)}.\]
Set $u':=u-\bar{\partial}^*_\varphi a$.
Then $\bar{\partial}u'=0$ and $\norm{u'}_{L^\infty(\affine)}\geq (1-C'_f\varepsilon)\norm{u}_{L^\infty(\affine)}$.
We choose $\varepsilon>0$ so that $1-C'_f\varepsilon>0$.
We set $V:=\{u'|u\in W\}$. Then $V\subset H_f$ and 
\[ \dim_\affine V=\dim_\affine W \geq (N+1)\int_\Lambda |df|^2dxdy -C_{\varepsilon}R.\]
For $u\in W$ (recall $\supp(u)\subset \Lambda$)
\[ \norm{u'}_{L^\infty(\affine)} \leq (1+C'_f\varepsilon)\norm{u}_{L^\infty(\affine)}
  = (1+C'_f\varepsilon)\norm{u}_{L^\infty(\Lambda)}, \]
\[ \norm{u'}_{L^\infty(\Lambda)}\geq (1-C'_f\varepsilon) \norm{u}_{L^\infty(\Lambda)} .\]
Hence 
\[ \norm{u'}_{L^\infty(\affine)} \leq \frac{1+C'_f\varepsilon}{1-C'_f\varepsilon}\norm{u'}_{L^\infty(\Lambda)}.\]
We choose $\varepsilon>0$ so small that 
\[ \frac{1+C'_f\varepsilon}{1-C'_f\varepsilon} \leq 2.\]
\end{proof}

\section{Infinite gluing: proof of Theorem \ref{thm: there are many non-degenerate curves}}
\label{section: infinite gluing}
We prove Theorem \ref{thm: there are many non-degenerate curves} in this section.
Our method is gluing:
We glue infinitely many rational curves to a (possibly degenerate) Brody curve $f:\affine\to \affine P^N$, 
and construct a non-degenerate one.

A kind of ``infinite gluing construction'' is classically used for the proof of Mittag-Leffler's theorem.
Probably another origin of infinite gluing construction is the shadowing lemma in dynamical system theory 
(for example, see Bowen \cite[Chapter 3]{Bowen}).
Angenent \cite{Angenent} developed a shadowing lemma for an elliptic PDE.
Gromov \cite[p. 403]{Gromov}
suggested an idea of gluing infinitely many rational curves to a (pseudo-)holomorphic curve.
Macr\`{i}-Nolasco-Ricciardi \cite{Macri-Nolasco-Ricciardi} developed gluing 
infinitely many selfdual vortices.
Gournay \cite{Gournay-thesis, Gournay-interpolation} studied 
an infinite gluing method for pseudo-holomorphic curves.
Tsukamoto \cite{Tsukamoto-Nagoya-gluing, Tsukamoto-MPAG} studied gluing infinitely many Yang-Mills 
instantons.

First we establish a result on gluing one rational curve:
\begin{proposition} \label{prop: gluing one rational curve}
There are $\delta_0>0$, $R_0>0$ and $K>0$ satisfying the following statement.
Let $f:\affine \to \affine P^N$ be a Brody curve.
If $f$ satisfies $\norm{df}_{L^\infty(D_R(p))} < \delta_0$ for some $p\in \affine$ and $R\geq R_0+1$, 
then there exists a holomorphic curve $g:\affine \to \affine P^N$ satisfying the following three conditions.

\noindent 
(i) $\delta_0\leq \norm{dg}_{L^\infty(D_R(p))}\leq 2/3$.

\noindent 
(ii) $||dg|(z)-|df|(z)|\leq K/|z-p|^3$ over $|z-p|>R$.

\noindent 
(iii) $d(f(z),g(z))\leq K/|z-p|^3$ for $z\neq p$.
\end{proposition}
\begin{proof}
The proof is just a calculation.
It may be helpful for some readers to consider the case of $N=1$ by themselves.
Let $\varepsilon>0$ be a sufficiently small number.
$\delta_0$, $R_0$, $K$ and $\varepsilon$ will be fixed later.
Several conditions will be imposed on them through the argument, but basically they need to satisfy 
\[ \delta_0\ll \frac{\varepsilon}{R_0}, \quad R_0\gg 1, \quad \varepsilon\ll \frac{1}{R_0^4}.\]

Fix $a>0$ so that the curve $q:\affine\to \affine P^N$ defined by 
$q(z) := [1:a/z^3:\dots:a/z^3]$ satisfies 
$\norm{dq}_{L^\infty(\affine)} = 1/12$.
Here 
\[ |dq|(z) = \frac{3a\sqrt{N}r^2}{\sqrt{\pi}(r^6+Na^2)}\quad (r=|z|).\]
We can suppose $\norm{dq}_{L^\infty(D_{R_0}(0))}=1/12$ since we choose $R_0\gg 1$.

From the symmetry we can assume $p=0$ and $f(0) = [1:0:\dots:0]$.
Let $f(z) = [1:f_1(z):\dots:f_N(z)]$ where $f_i(z)$ are meromorphic functions in $\affine$.
Since $|df|\leq \delta_0$ over $|z|\leq R$ with $R\geq R_0+1$, 
if we choose $\delta_0$ sufficiently small ($\delta_0\ll \varepsilon/R_0$), 
we have 
\begin{equation}\label{eq: f_i are smallon R_0 disk}
 |f_i(z)|\leq \varepsilon, \quad  |f_i'(z)| \leq \varepsilon \quad (|z|\leq R_0).
\end{equation}

Set $g_i(z) := f_i(z) +a/z^3$, and we define $g:\affine \to \affine P^N$ by 
$g(z):=[1:g_1(z):\dots:g_N(z)]$.
We will prove that this map $g$ satisfies the conditions (i), (ii), (iii).

First we study the condition (iii).
The Fubini-Study metric is given by 
\[ ds^2 = 
 \frac{\sum_{i=1}^N |dz_i|^2 + \sum_{1\leq i<j\leq N}|z_j dz_i-z_i dz_j|^2}{\pi(1+\sum |z_i|^2)^2} \quad
   \text{on $\{[1:z_1:\dots:z_N]\}$}.\]
\begin{equation*}
  \begin{split}
  ds^2 \leq \frac{\sum |dz_i|^2 + 2(\sum |z_i|^2)(\sum |dz_i|^2)}{\pi(1+\sum |z_i|^2)^2} 
       \leq \frac{2(1+\sum |z_i|^2)\sum |dz_i|^2}{\pi(1+\sum |z_i|^2)^2} \leq \frac{2}{\pi}\sum |dz_i|^2
  \end{split}
\end{equation*}
Hence $ds\leq \sqrt{2/\pi}\sqrt{\sum_{i=1}^N |dz_i|^2}$.
Thus for $f(z) =[1:f_1(z):\dots:f_N(z)]$ and $g(z)=[1:f_1(z)+a/z^3:\dots:f_N(z)+a/z^3]$
we get 
\begin{equation} \label{eq: condition (iii) for gluing}
  d(f(z), g(z))\leq \sqrt{2/\pi}\sqrt{\sum_{i=1}^N |a/z^3|^2} = \frac{a\sqrt{2N/\pi}}{|z|^3}.
\end{equation}

Next we study the conditions (i) and (ii).
We have 
\[ |df|(z) = 
\frac{\sqrt{\sum |f_i'(z)|^2 + \sum_{i<j}|f_i'(z)f_j(z)-f_i(z)f_j'(z)|^2}}{\sqrt{\pi}(1+\sum |f_i(z)|^2)}, \]
\[ |dg|(z) = 
\frac{\sqrt{\sum |g_i'(z)|^2 + \sum_{i<j}|g_i'(z)g_j(z)-g_i(z)g_j'(z)|^2}}{\sqrt{\pi}(1+\sum |g_i(z)|^2)}, \]
where 
\[ g_i' = f_i' -\frac{3a}{z^4},\quad 
   g_i'g_j-g_ig_j' = (f_i'f_j-f_if_j') +\frac{3a}{z^4}(f_i-f_j) + \frac{a}{z^3}(f_i'-f_j').\]

Case 1: Suppose $r:=|z|\leq R_0$. 
We will prove $\delta_0\leq \norm{dg}_{L^\infty(D_{R_0}(0))}\leq 2/3$.
From (\ref{eq: f_i are smallon R_0 disk}),
\[ |g_i(z)|\leq \varepsilon + \frac{a}{r^3} \leq \frac{2a}{r^3},\quad 
   |g_i'(z)|\geq \frac{3a}{r^4}-\varepsilon \geq \frac{3a}{2r^4}.\]
Here we have supposed $\varepsilon \leq \min(a/R_0^3, 3a/(2R_0^4))$.
Then 
\[ |dg|(z) \geq \frac{\sqrt{N}(3a/(2r^4))}{\sqrt{\pi}(1+4Na^2/r^6)}
    = \frac{3a\sqrt{N}r^2}{2\sqrt{\pi}(r^6+4Na^2)}\geq \frac{3a\sqrt{N}r^2}{8\sqrt{\pi}(r^6+Na^2)}
    = \frac{|dq|(z)}{8}.\]
Hence $\norm{dg}_{L^\infty(D_{R_0}(0))}\geq (1/8)\norm{dq}_{L^\infty(D_{R_0}(0))} = 1/96 \geq \delta_0$.
(Here we have supposed $\delta_0\leq 1/96$.)
On the other hand, 
\[ |dg|(z) = \frac{\sqrt{\sum |3az^2-z^6f_i'|^2 + \sum_{i<j}|z^6(f_i'f_j-f_j'f_i)+3az^2(f_i-f_j)+az^3(f_i'-f_j')|^2}}
{\sqrt{\pi}(r^6+\sum |a+z^3f_i|^2)}.\]
From (\ref{eq: f_i are smallon R_0 disk}),
\[ |a+z^3f_i|\geq a -\varepsilon R_0^3\geq \frac{a}{2},\quad (\text{here we suppose $\varepsilon R_0^3\leq a/2$}).\]
\[ r^6+\sum |a+z^3f_i|^2 \geq r^6+\frac{Na^2}{4}\geq \frac{r^6+Na^2}{4}.\]
\[ |3az^2-z^6f_i'|\leq 3ar^2+r^6\varepsilon \leq r^2(3a+R_0^4\varepsilon) \leq 4ar^2, \quad
   (\text{we suppose $R_0^4\varepsilon\leq a$}).\] 
\[ |z^6(f_i'f_j-f_j'f_i)+3az^2(f_i-f_j)+az^3(f_i'-f_j')| \leq r^2(2\varepsilon^2R_0^4+6a\varepsilon +2a\varepsilon R_0)
   \leq \frac{ar^2}{\sqrt{\binom{N}{2}}}.\]
Here we have supposed $2\varepsilon^2R_0^4+6a\varepsilon +2a\varepsilon R_0\leq a/\sqrt{\binom{N}{2}}$.
Then 
\[ |dg|(z) \leq \frac{4ar^2\sqrt{16N+1}}{\sqrt{\pi}(r^6+Na^2)}\leq 
    \frac{24ar^2\sqrt{N}}{\sqrt{\pi}(r^6+Na^2)} = 8|dq|(z)\leq \frac{2}{3}, \quad 
    (|dq|\leq 1/12).\]
Thus we get $\delta_0\leq \norm{dg}_{L^\infty(D_{R_0}(0))}\leq 2/3$.

Case 2: Suppose $|z|\geq R_0$.
We will prove $||df|(z)-|dg|(z)|\leq K/r^3$ for an appropriate $K>0$.
We have 
\[ \left||f_i|^2-|g_i|^2\right|\leq (|f_i|+|g_i|) \cdot |f_i-g_i|
   \leq (2|f_i|+a/r^3)(a/r^3)\leq (2|f_i|+a/R_0^3)(a/r^3), \]
\[ \sum \left||f_i|^2-|g_i|^2\right| \leq \frac{a}{r^3}\left(2\sum |f_i| + \frac{Na}{R_0^3}\right)\leq 
   \frac{2a}{r^3}\left(1+\sum |f_i|\right) \quad (\text{we suppose $\frac{Na}{R_0^3}\leq 2$}).\] 
If $|f_i|\geq a/r^3$, then
\[ |g_i|^2 \geq \left(|f_i|-a/r^3\right)^2 \geq \frac{|f_i|^2}{2} -\frac{a^2}{r^6}
    \geq \frac{|f_i|^2}{2}-\frac{a^2}{R_0^6}, \quad ((x-y)^2\geq \frac{x^2}{2}-y^2).\]
If $|f_i| < a/r^3$, then 
\[ |g_i|^2\geq 0 >  \frac{|f_i|^2}{2} -\frac{a^2}{r^6} \geq \frac{|f_i|^2}{2}-\frac{a^2}{R_0^6}.\]
Therefore we always have $|g_i|^2 \geq |f_i|^2/2 -a^2/R_0^6$.
\[ 1+\sum |g_i|^2 \geq \left(1-\frac{Na^2}{R_0^6}\right)+\frac{1}{2}\sum |f_i|^2
   \geq \frac{1}{2}\left(1+\sum |f_i|^2\right)\quad 
(\text{we suppose $\frac{Na^2}{R_0^6}\leq \frac{1}{2}$}).\]
Hence 
\begin{equation} \label{eq: difference of denominators}
  \begin{split}
   \left|\frac{1}{1+\sum |g_i|^2}-\frac{1}{1+\sum |f_i|^2}\right|
   &\leq \frac{\frac{4a}{r^3}\left(1+\sum |f_i|\right)}{\left(1+\sum |f_i|^2\right)^2}
   \leq \frac{4a\sqrt{N+1}\sqrt{1+\sum |f_i|^2}}{r^3 \left(1+\sum |f_i|^2\right)^2}\\
   &= \frac{4a\sqrt{N+1}}{r^3\left(1+\sum |f_i|^2\right)^{3/2}}
   \leq \frac{4a\sqrt{N+1}}{r^3(1+\sum |f_i|^2)}.
  \end{split}
\end{equation}
Then, from $g_i' = f_i' - 3a/z^4$ and the above (\ref{eq: difference of denominators}), 
\begin{equation*}
 \begin{split}
 \left|\frac{|g_i'|}{1+\sum |g_k|^2}-\frac{|f_i'|}{1+\sum |f_k|^2}\right|
 &\leq \left|\frac{|g_i'|}{1+\sum |g_k|^2}-\frac{|g_i'|}{1+\sum |f_k|^2}\right|
 + \left|\frac{|g'_i|}{1+\sum |f_k|^2} -\frac{|f_i'|}{1+\sum |f_k|^2}\right| \\
 &\leq \frac{4a\sqrt{N+1}(|f_i'|+3a/r^4)}{r^3(1+\sum |f_k|^2)} 
   + \frac{3a}{r^4(1+\sum |f_k|^2)}.
 \end{split}
\end{equation*}
From $|df|\leq 1$, we have $|f_i'|/(1+\sum |f_k|^2) \leq \sqrt{\pi}$.
Hence the above is bounded by 
\[ \frac{4a\sqrt{N+1}}{r^3}(\sqrt{\pi}+3a/r^4) +3a/r^4
   \leq \frac{4a\sqrt{N+1}}{r^3}(\sqrt{\pi}+3a)+\frac{3a}{r^3} .\]
Here we have supposed $r\geq R_0\geq 1$.
Set $K_a := 4a\sqrt{N+1}(\sqrt{\pi}+3a) + 3a$. Then 
\begin{equation} \label{eq: term of difference in gluing}
   \left|\frac{|g_i'|}{1+\sum |g_k|^2}-\frac{|f_i'|}{1+\sum |f_k|^2}\right| \leq 
   \frac{K_a}{r^3}.
\end{equation}

From (\ref{eq: difference of denominators}), for $i<j$, 
\begin{equation*}
  \begin{split}
   \left|\frac{|g_i'g_j-g_j'g_i|}{1+\sum |g_k|^2}-\frac{|f_i'f_j-f_j'f_i|}{1+\sum |f_k|^2}\right|
   &\leq \left|\frac{|g_i'g_j-g_j'g_i|}{1+\sum |g_k|^2}-\frac{|g_i'g_j-g_j'g_i|}{1+\sum |f_k|^2}\right|
   + \left|\frac{|g_i'g_j-g_j'g_i|}{1+\sum |f_k|^2}-\frac{|f_i'f_j-f_j'f_i|}{1+\sum |f_k|^2}\right| \\
   &\leq \frac{4a\sqrt{N+1} |g_i'g_j-g_j'g_i|}{r^3(1+\sum |f_k|^2)}
    + \frac{|(g_i'g_j-g_j'g_i)-(f_i'f_j-f_j'f_i)|}{1+\sum |f_k|^2}
  \end{split}
\end{equation*}
From $g_i'g_j-g_j'g_i= (f_i'f_j-f_j'f_i)+(3a/z^4)(f_i-f_j)+(a/z^3)(f_i'-f_j')$, this is bounded by 
\begin{equation} \label{eq: correlation term of difference}
  \begin{split}
   \frac{4a\sqrt{N+1}}{r^3}&\left(\frac{|f_i'f_j-f_j'f_i|}{1+\sum |f_k|^2} 
   + \frac{3a(|f_i|+|f_j|)}{r^4(1+\sum |f_k|^2)} 
   + \frac{a(|f_i'|+|f_j'|)}{r^3(1+\sum |f_k|^2)}\right) \\
   &+ \frac{3a(|f_i|+|f_j|)}{r^4(1+\sum |f_k|^2)} 
   + \frac{a(|f_i'|+|f_j'|)}{r^3(1+\sum |f_k|^2)}.
  \end{split}
\end{equation}
From $|df|\leq 1$, 
\[ \frac{|f_i'f_j-f_j'f_i|}{1+\sum |f_k|^2} \leq \sqrt{\pi}, \quad 
   \frac{|f_i'|+|f_j'|}{1+\sum |f_k|^2} \leq 2\sqrt{\pi}.\]
Since $i<j$,
\[ \frac{|f_i|+|f_j|}{1+\sum |f_k|^2} \leq \frac{\sqrt{2}\sqrt{|f_i|^2+|f_j|^2}}{1+\sum |f_k|^2}\leq \sqrt{2}.\]
Hence the above (\ref{eq: correlation term of difference}) is bounded by 
\begin{equation*}
 \begin{split}
  \frac{4a\sqrt{N+1}}{r^3}\left(\sqrt{\pi}+\frac{3a\sqrt{2}}{r^4}+\frac{2a\sqrt{\pi}}{r^3}\right)
  + \frac{3a\sqrt{2}}{r^4} + \frac{2a\sqrt{\pi}}{r^3} \\
  \leq \frac{4a\sqrt{N+1}}{r^3}(\sqrt{\pi}+3a\sqrt{2}+2a\sqrt{\pi})
   + \frac{3a\sqrt{2}}{r^3} + \frac{2a\sqrt{\pi}}{r^3}.
 \end{split}
\end{equation*}
Here $r\geq R_0\geq 1$.
Set $K'_a := 4a\sqrt{N+1}(\sqrt{\pi}+3a\sqrt{2}+2a\sqrt{\pi}) + 3a\sqrt{2}+2a\sqrt{\pi}$.
Then 
\[  \left|\frac{|g_i'g_j-g_j'g_i|}{1+\sum |g_k|^2}-\frac{|f_i'f_j-f_j'f_i|}{1+\sum |f_k|^2}\right|
    \leq \frac{K'_a}{r^3}.\] 
From this and (\ref{eq: term of difference in gluing}),
\[ \left||dg|(z)-|df|(z)\right| \leq (1/\sqrt{\pi})\sqrt{N(K_a/r^3)^2+\binom{N}{2}(K'_a/r^3)^2}
   = \frac{\sqrt{N K_a^2+\binom{N}{2}(K_a')^2}}{\sqrt{\pi}r^3}.\]
Here we have used the inequality 
\[ \left|\sqrt{x_1^2+\dots+x_l^2}-\sqrt{y_1^2+\dots+y_l^2}\right|  \leq 
    \sqrt{(x_1-y_1)^2+\dots+(x_l-y_l)^2}.\]
Set 
\[ K:= \max\left(a\sqrt{2N/\pi}, \sqrt{NK_a^2+\binom{N}{2}(K'_a)^2}/\sqrt{\pi}\right). \]
(This $K$ satisfies the condition (iii) by (\ref{eq: condition (iii) for gluing}).)
Then
\[ \left||df|(z)-|dg|(z)\right| \leq \frac{K}{r^3} \quad (r\geq R_0).\]
Thus we have proved the condition (ii).

For $R_0\leq |z|\leq R$, 
\[ |dg|(z)\leq \norm{df}_{L^\infty(D_R(0))} + \frac{K}{R_0^3}\leq \delta_0+\frac{1}{2}\leq \frac{2}{3},\]
where we have chosen $R_0$ and $\delta_0$ so that $K/R_0^3\leq 1/2$ and $\delta_0\leq 1/6$.
In Case 1, we proved $\delta_0\leq \norm{dg}_{L^\infty(D_{R_0}(0))}\leq 2/3$.
Thus we get the condition (i):
\[ \delta_0\leq \norm{dg}_{L^\infty(D_R(0))}\leq 2/3.\]
\end{proof}

\begin{proof}[Proof of Theorem \ref{thm: there are many non-degenerate curves}]
Let $\norm{df}_{L^\infty(\affine)}\leq 1-\tau$, $(0<\tau\leq 1)$.
Let $\delta_0$, $R_0$, $K$ be the positive numbers introduced in Proposition \ref{prop: gluing one rational curve}.
For $\varepsilon>0$, we set $\delta :=\min(\delta_0, \sqrt{\varepsilon})$.
Let $R=R(\varepsilon, \tau)\geq R_0+1$ be a large positive number which will be fixed later.

We index the elements of $\mathbb{Z}^2$ by natural numbers:
$\mathbb{Z}^2 = \{(\alpha_1, \beta_1), (\alpha_2,\beta_2),(\alpha_3,\beta_3),\dots\}$.
For $n\geq 1$, we set 
$p_n := 2R(\alpha_n+\sqrt{-1}\beta_n)$ and 
$\Lambda_n := \{x+y\sqrt{-1}\in \affine |\,  |x-2R\alpha_n|\leq R, |y-2R\beta_n|\leq R\}$.
The squares $\Lambda_n$ $(n\geq 1)$ give a tiling of the plane $\affine$.

We inductively define the sequence of Brody curves $f_n:\affine\to \affine P^N$ $(n\geq 0)$ as follows.
We set $f_0:=f$.
Suppose we have defined $f_n$.

\begin{enumerate}
 \item
   If $\norm{df}_{L^\infty(\Lambda_{n+1})}\geq \delta$, then we set $f_{n+1}:=f_n$.
 \item
  If $\norm{df}_{L^\infty(\Lambda_{n+1})} < \delta$ 
  and $\norm{d f_n}_{L^\infty(\Lambda_{n+1})}\geq \delta_0$, then we set 
  $f_{n+1}:=f_n$.
  \item 
  If $\norm{df}_{L^\infty(\Lambda_{n+1})}< \delta$ and $\norm{df_{n}}_{L^\infty(\Lambda_{n+1})} < \delta_0$,
 then we apply Proposition \ref{prop: gluing one rational curve}
 to $f_n$ and $p_{n+1}$ (note $D_R(p_{n+1})\subset \Lambda_{n+1}$)
 and get a holomorphic map $f_{n+1}: \affine \to \affine P^N$ satisfying the following (i), (ii), (iii).
 
\noindent 
(i) $\delta_0\leq \norm{df_{n+1}}_{L^\infty(D_R(p_{n+1}))}\leq 2/3$.

\noindent 
(ii) $\left||df_{n+1}|(z) - |df_n|(z)\right|\leq K/|z-p_{n+1}|^3$ over $|z-p_{n+1}| > R$.

\noindent 
(iii) $d(f_n(z), f_{n+1}(z))\leq K/|z-p_{n+1}|^3$ for $z\neq p_{n+1}$.
\end{enumerate}
For every $n\geq 1$, by (i) and (ii) 
\[ |df_n|(z)\leq \max(1-\tau, 2/3) +\sum_{k: |z-p_k|>R} \frac{K}{|z-p_k|^3} 
    \leq \max(1-\tau, 2/3) + \frac{\const\cdot K}{R^3}.\]
Here $\const$ is a positive constant independent of $n$.
We choose $R$ so large that the right hand side is bounded by 
$\max(1-\tau/2, 3/4) <1$.
Then all $f_n:\affine \to \affine P^N$ become Brody curves, and we can continue the above inductive 
construction infinitely many times.
Moreover, for all $n\geq 1$,
\begin{equation}  \label{eq: bound on spherical derivative of f_n}
  \norm{df_n}_{L^\infty(\affine)}\leq \max(1-\tau/2,3/4).
\end{equation}

For any compact set $\Omega\subset \affine$, by the condition (iii), there exists $n(\Omega)\geq 1$ such that 
\[ \sum_{n\geq n(\Omega)} \sup_{z\in \Omega}d(f_n(z),f_{n+1}(z)) 
   \leq \sum_{k: d(p_k, \Omega)\geq 1} \frac{K}{d(p_k,\Omega)^3} <+\infty.\]
Hence the sequence $f_n$ converges to a holomorphic curve $g:\affine \to \affine P^N$
uniformly over every compact subset of $\affine$.
From (\ref{eq: bound on spherical derivative of f_n}) we have 
$\norm{dg}_{L^\infty(\affine)}\leq \max(1-\tau/2,3/4)<1$.
We will prove that $g$ is non-degenerate and $\rho(g)\geq \rho(f)-\varepsilon$.

For proving the non-degeneracy of $g$, 
it is enough to show $\norm{dg}_{L^\infty(\Lambda_n)}\geq \delta/2$
for all $n\geq 1$. (See the condition (ii) of Definition-Lemma \ref{defintion-lemma: non-degenerate}.)

Case 1: If $|df|(z)\geq \delta$ for some $z\in \Lambda_n$, then 
\[ |dg|(z) \geq \delta -\sum_{k: k\neq n}\frac{K}{|z-p_k|^3} \geq \delta -\frac{\const\cdot K}{R^3}.\]
We can choose $R$ so large that $\norm{dg}_{L^\infty(\Lambda_n)}\geq \delta/2$.

Case 2: If $|df|(z) < \delta$ for all $z\in \Lambda_n$, then for some $k\in \{n-1, n\}$ and $w\in \Lambda_n$
we have $|df_k|(w)\geq \delta_0$.
Hence 
\[ |dg|(w) \geq \delta_0 -\sum_{l: l\neq n}\frac{K}{|w-p_l|^3} \geq \delta -\frac{\const\cdot K}{R^3} .\]
We can choose $R$ so large that $\norm{dg}_{L^\infty(\Lambda_n)}\geq \delta/2$.

We have proved that $g$ is non-degenerate.
Next we will prove $\rho(g)\geq \rho(f)-\varepsilon$.
For this sake, it is enough to prove that for every $n\geq 1$
\begin{equation}  \label{eq: step to rho(g) geq rho(f)-varepsilon}
 \frac{1}{(2R)^2}\int_{\Lambda_n} |dg|^2dxdy \geq \frac{1}{(2R)^2}\int_{\Lambda_n} |df|^2dxdy -\varepsilon.
\end{equation}

Case 1: If $\norm{df}_{L^\infty(\Lambda_n)}\geq \delta$, then for all $z\in \Lambda_n$
\begin{equation*}
   \left| |dg|^2(z) -|df|^2(z)\right| \leq 2\left| |dg|(z)-|df|(z)\right|  \leq 
   \sum_{k:k\neq n}\frac{2K}{|z-p_k|^3} \leq \frac{\const\cdot K}{R^3} \leq \varepsilon
\end{equation*}
for sufficiently large $R$.
Hence (\ref{eq: step to rho(g) geq rho(f)-varepsilon}) holds if we choose $R$ sufficiently large.

Case 2: If $\norm{df}_{L^\infty(\Lambda_n)} < \delta$, then (recall $\delta=\min(\delta_0,\sqrt{\varepsilon})$)
\[ \frac{1}{(2R)^2}\int_{\Lambda_n}|df|^2 dxdy \leq \delta^2 \leq \varepsilon.\]
Hence (\ref{eq: step to rho(g) geq rho(f)-varepsilon}) holds trivially.

Thus we have proved $\rho(g)\geq \rho(f)-\varepsilon$.
\end{proof}

\vspace{10mm}

\address{ Shinichiroh Matsuo \endgraf
Department of Mathematics, Kyoto University, Kyoto 606-8502, Japan}

\textit{E-mail address}: \texttt{matsuo@math.kyoto-u.ac.jp}

\vspace{0.5cm}

\address{ Masaki Tsukamoto \endgraf
Department of Mathematics, Kyoto University, Kyoto 606-8502, Japan}

\textit{E-mail address}: \texttt{tukamoto@math.kyoto-u.ac.jp}

\end{document}